\documentclass[12pt,a4paper]{amsart}

\makeatletter
\tagsleft@false
\makeatother

\usepackage{mathrsfs}
\usepackage{tikz}
\usepackage{amsfonts}
\usepackage{amsmath}
\usepackage{amssymb}
\usepackage{enumerate}
\usepackage{hyperref}
\usepackage{latexsym}
\usepackage{xcolor}
\usepackage[shortlabels]{enumitem}
\usepackage[numbers,sort&compress]
{natbib}
\setlist[enumerate,1]{label={\upshape(\roman*)}}

 \makeatletter
    
    \newcommand{\Rmnum}[1]
    {\expandafter\@slowromancap\romannumeral #1@}
    \makeatother

\newtheorem{thm}{Theorem}[section]

\newtheorem{prop}[thm]{Proposition}
\newtheorem{lemma}[thm]{Lemma}

\newtheorem{cor}[thm]{Corollary}

\newtheorem{example}[thm]{Example}
\newtheorem{defin}[thm]{Definition}

\theoremstyle{definition}
\newtheorem{remark}[thm]{Remark}

\def\wz{\tilde}

\title[Locally semicomplete weakly distance-regular digraphs]{Locally semicomplete weakly distance-regular digraphs}

\date{}

\thanks{*Corresponding author}

\author[Yang]{Yuefeng Yang}
\address{School of Science\\China University of Geosciences\\Beijing 100083\\China}
\email{yangyf@cugb.edu.cn}

\author[Li]{Shuang Li*}
\address{Laboratory of Mathematics and Complex Systems (MOE),~School of Mathematical Sciences\\Beijing Normal University\\Beijing 100875\\China}
\email{lishuangyx@mail.bnu.edu.cn}

\author[Wang]{Kaishun Wang}
\address{Laboratory of Mathematics and Complex Systems (MOE),~School of Mathematical Sciences\\Beijing Normal University\\Beijing 100875\\China}
\email{wangks@bnu.edu.cn}

\begin{document}

\begin{abstract}
A digraph is semicomplete if any two vertices are connected by at least one arc and is locally semicomplete if the out-neighbourhood (resp. in-neighbourhood) of any vertex induces a semicomplete digraph. In this paper,
we characterize all locally semicomplete weakly distance-regular digraphs under the assumption of commutativity.
\end{abstract}



\maketitle
\section{Introduction}

A \emph{digraph} $\Gamma$ is a pair $(V(\Gamma),A(\Gamma))$, where $V(\Gamma)$ is a finite nonempty set of vertices and $A(\Gamma)$ is a set of ordered pairs ({\em arcs}) $(x,y)$ with distinct vertices $x$ and $y$. A subdigraph of $\Gamma$ induced by a subset $U\subseteq V(\Gamma)$ is denoted by $\Gamma[U]$. For any arc $(x,y)\in A(\Gamma)$, if $A(\Gamma)$ also contains an arc $(y,x)$, then $\{(x,y),(y,x)\}$ can be viewed as an {\em edge}. We say that $\Gamma$ is an \emph{undirected graph} or a {\em graph} if $A(\Gamma)$ is a symmetric relation. A vertex $x$ is {\em adjacent} to $y$ if $(x,y)\in A(\Gamma)$. In this case, we also call $y$ an \emph{out-neighbour} of $x$, and $x$ an \emph{in-neighbour} of $y$. The set of all out-neighbours of $x$ is denoted by $N_{\Gamma}^{+}(x)$, while the set of in-neighbours is denoted by $N_{\Gamma}^{-}(x)$. If no confusion occurs, we write $N^{+}(x)$ (resp. $N^{-}(x)$) instead of $N_{\Gamma}^{+}(x)$ (resp. $N_{\Gamma}^{-}(x)$). A digraph is said to be {\em regular of valency} $k$ if the number of in-neighbour and out-neighbour of all vertices are equal to $k$. The
\emph{adjacency matrix} $A$ of $\Gamma$ is the $|V(\Gamma)|\times |V(\Gamma)|$ matrix whose $(x,y)$-entry is $1$ if $y\in N^{+}(x)$, and
$0$ otherwise. 

Given digraphs $\Gamma$ and $\Sigma_x$ for every $x\in V(\Gamma)$, a {\em generalized lexicographic product} of $\Gamma$ and $(\Sigma_x)_{x\in V(\Gamma)}$, denoted by $\Gamma\circ(\Sigma_x)_{x\in V(\Gamma)}$ is the digraph with the vertex set $\cup_{x\in V(\Gamma)}(\{x\}\times V(\Sigma_x))$ where $((x,u_x),(y,v_y))\in A(\Gamma\circ(\Sigma_x)_{x\in V(\Gamma)})$ whenever $(x,y)\in A(\Gamma)$, or $x=y$ and $(u_x,v_x)\in A(\Sigma_x)$. If $\Sigma_x=\Sigma$ for all $x\in V(\Gamma)$, then $\Gamma\circ(\Sigma_x)_{x\in V(\Gamma)}$
becomes the standard {\em lexicographic product} $\Gamma\circ\Sigma$. 

A \emph{path} of length $r$ from $x$ to $y$ in the digraph $\Gamma$ is a finite sequence of vertices $(x=w_{0},w_{1},\ldots,w_{r}=y)$ such that $(w_{t-1}, w_{t})\in A(\Gamma)$ for $1\leq t\leq r$. A digraph (resp. graph) is said to be \emph{strongly connected} (resp. \emph{connected}) if, for any vertices $x$ and $y$, there is a path from $x$ to $y$. A path $(w_{0},w_{1},\ldots,w_{r-1})$ is called a \emph{circuit} of length $r$ when $(w_{r-1},w_0)\in A\Gamma$. The \emph{girth} of $\Gamma$ is the length of a shortest circuit in $\Gamma$. The length of a shortest path from $x$ to $y$ is called the \emph{distance} from $x$ to $y$ in $\Gamma$, denoted by $\partial_\Gamma(x,y)$. The maximum value of the distance function in $\Gamma$ is called the \emph{diameter} of $\Gamma$. Let $\wz{\partial}_{\Gamma}(x,y):=(\partial_{\Gamma}(x,y),\partial_{\Gamma}(y,x))$ be the \emph{two-way distance} from $x$ to $y$, and $\wz{\partial}(\Gamma):=\{\wz{\partial}_{\Gamma}(x,y)\mid x,y\in V(\Gamma)\}$ the \emph{two-way distance set} of $\Gamma$. If no confusion occurs, we write $\partial(x,y)$ (resp. $\tilde{\partial}(x,y)$) instead of $\partial_\Gamma(x,y)$ (resp. ${\tilde{\partial}}_\Gamma(x,y)$).  For any $\wz{i}:=(a,b)\in\wz{\partial}(\Gamma)$, we define $\Gamma_{\wz{i}}$ to be the set of ordered pairs $(x,y)$ with $\wz{\partial}(x,y)=\wz{i}$, and write $\Gamma_{a,b}$ instead of $\Gamma_{(a,b)}$. An arc $(x,y)$ of $\Gamma$ is of type $(1,r)$ if $\partial(y,x)=r$.

In \cite{KSW03}, the third author and Suzuki proposed a natural directed version of a distance-regular graph (see \cite{AEB98,DKT16} for a background of the theory of distance-regular graphs) without bounded diameter, i.e., a weakly distance-regular digraph. A strongly connected digraph $\Gamma$ is said to be \emph{weakly distance-regular} if, for any $\wz{h}$, $\wz{i}$, $\wz{j}\in\wz{\partial}(\Gamma)$, the number of $z\in V(\Gamma)$ such that $\wz{\partial}(x,z)=\wz{i}$ and $\wz{\partial}(z,y)=\wz{j}$ is constant whenever $\wz{\partial}(x,y)=\wz{h}$.  This constant is denoted by $p_{\wz{i},\wz{j}}^{\wz{h}}(\Gamma)$. The integers
$p_{\wz{i},\wz{j}}^{\wz{h}}(\Gamma)$ are called the \emph{intersection numbers} of $\Gamma$. If no confusion occurs, we write $p_{\wz{i},\wz{j}}^{\wz{h}}$ instead of $p_{\wz{i},\wz{j}}^{\wz{h}}(\Gamma)$. We say
that $\Gamma$ is \emph{commutative} if $p_{\tilde{i},\tilde{j}}^{\tilde{h}}=p_{\tilde{j},\tilde{i}}^{\tilde{h}}$ for all $\tilde{i},\tilde{j},\tilde{h}\in\wz{\partial}(\Gamma)$. For more information about weakly distance-regular digraphs, see \cite{YF22,HS04,AM,KSW03,KSW04,YYF16,YYF18,YYF20,YYF22,YYF,QZ23,QZ24}. 

A digraph $\Gamma$ is {\em semicomplete}, if for any pair of vertices $x,y\in V(\Gamma)$, either $(x,y)\in A(\Gamma)$, or $(y,x)\in A(\Gamma)$, or both. A digraph $\Gamma$ is {\em locally semicomplete}, if $\Gamma[N^+(x)]$ and $\Gamma[N^-(x)]$ are both semicomplete for every vertex $x$ of $\Gamma$. Note that a semicomplete digraph is also locally semicomplete. Locally semicomplete digraphs were introduced in 1990 by Bang-Jensen \cite{JBJ90}. If a semicomplete weakly distance-regular digraph is not a complete graph, then it has diameter $2$ and girth $g\leq3$ (see \cite[Proposition 3.5]{YYF}).


In this paper, we discuss locally semicomplete weakly distance-regular digraphs. To state our main result, we need additional notations and terminologies. 
An \emph{$(r,m)$-team tournament} is a digraph obtained from the complete multipartite graph with $r$ independent sets of size $m$ by replacing every edge $\{(x,y),(y,x)\}$ by exactly one of the arcs $(x,y)$ or $(y,x)$. A regular $(r,m)$-team tournament with adjacency matrix $A$ is said to be a {\em doubly regular $(r,m)$-team tournament with parameters $(\alpha,\beta,\gamma)$} if $A^2=\alpha A+\beta A^t+\gamma(J-I-A-A^t)$, where $J$ is the matrix with $1$ in every entry, and $I$ denotes the identity matrix.  \cite[Theorem 4.3]{JG14} implies that a doubly regular $(r,m)$-team tournament is of type \Rmnum{1}, \Rmnum{2} or \Rmnum{3}. Let $\Sigma$ be a doubly regular $(r,m)$-team tournament with parameters $(\alpha,\beta,\gamma)$, where $V(\Sigma)=V_1\cup \cdots \cup V_r$ is the partition of the vertex set into $r$ independent sets of size $m$. We say that $\Sigma$ is of {\em type \Rmnum{2}} if $\beta-\alpha=0$, $m$ is even and $|N^+(x)\cap V_i|=m/2$ for all $x\notin V_i$ and $i\in \{1,\cdots,r\}$.


The main theorem is as follows, which characterizes all locally semicomplete weakly distance-regular digraphs which are not semicomplete under the assumption of commutativity.




\begin{thm}\label{main1}
	Let $\Gamma$ be a commutative weakly distance-regular digraph. Then $\Gamma$ is locally semicomplete but not semicomplete if and only if $\Gamma$ is isomorphic to one of the following digraphs:
\begin{enumerate}
	\item\label{main1-4} $\Lambda \circ K_{m}$;
	
	\item\label{main1-1} ${\rm Cay}(\mathbb{Z}_{il},\{1,i\})\circ K_{m}$;
	
	\item\label{main1-2} ${\rm Cay}(\mathbb{Z}_{iq},\{1,i\})\circ(\Sigma_x)_{x\in\mathbb{Z}_{iq}}$.
\end{enumerate}
Here, $i\in \{1,2\}$, $m\geq1$, $q\geq4$, $l\geq 5-i$, $(\Sigma_x)_{x\in \mathbb{Z}_{iq}}$ are semicomplete weakly distance-regular digraphs of diameter $2$ with $p_{\wz{i},\wz{j}}^{\wz{h}}(\Sigma_0)=p_{\wz{i},\wz{j}}^{\wz{h}}(\Sigma_x)$ for all $x$ and $\wz{i},\wz{j},\wz{h}$, and $\Lambda$ is a doubly regular $(r,2)$-team tournament of type \Rmnum{2} for a positive integer $r$ with $4\mid r$.
\end{thm}

The remainder of this paper is organized as follows. In Section 2, we  introduce some basic results for weakly distance-regular digraphs.
In Sections 3 and 4, we give a characterization of mixed arcs of type $(1,q-1)$ in a locally semicomplete commutative weakly distance-regular digraph, and  divide it into two sections according to the value of $q$.
In Section 5, we determine special subdigraphs of a locally semicomplete commutative weakly distance-regular digraph based on the results in Sections 3 and 4.
In Section 6, we prove Theorem \ref{main1} based on the results in Section 5.

\section{Preliminaries}

In this section, we always assume that $\Gamma$ is a weakly distance-regular digraph. We shall give some results for $\Gamma$ which are used frequently in this paper.

Let $R=\{\Gamma_{\wz{i}}\mid\wz{i}\in\tilde{\partial}(\Gamma)\}$. Then $(V(\Gamma),R)$ is an association scheme (see \cite{EB21,EB84,PHZ96,PHZ05} for the theory of association schemes), which is called the \emph{attached scheme} of $\Gamma$. For each $\wz{i}:=(a,b)\in\wz{\partial}(\Gamma)$, we define $\wz{i}^{t}:=(b,a)$. Denote $\Gamma_{\wz{i}}(x)=\{y\in V(\Gamma)\mid\wz{\partial}(x,y)=\wz{i}\}$ and $P_{\wz{i},\wz{j}}(x,y)=\Gamma_{\wz{i}}(x)\cap\Gamma_{\wz{j}^{t}}(y)$ for all $\wz{i},\wz{j}\in\wz{\partial}(\Gamma)$ and $x,y\in V\Gamma$. The size of $\Gamma_{\tilde{i}}(x)$
depends only on $\tilde{i}$, and is denoted by $k_{\tilde{i}}$. For the sake of convenience, we write $k_{a,b}$ instead of $k_{(a,b)}$.

\begin{lemma}\label{jb}
{\rm (\cite[Chapter \Rmnum{2}, Proposition 2.2]{EB84} and \cite[Proposition 5.1]{ZA99})} Let $\tilde{d},\tilde{e},\tilde{f},\tilde{g}\in\tilde{\partial}(\Gamma)$. The following hold:
\begin{enumerate}
\item\label{jb-1} $k_{\wz{d}}k_{\wz{e}}=\sum_{\wz{h}\in\wz{\partial}(\Gamma)}p_{\wz{d},\wz{e}}^{\wz{h}}k_{\wz{h}}$;

\item\label{jb-2}  $p_{\wz{d},\wz{e}}^{\wz{f}}k_{\wz{f}}=p_{\wz{f},\wz{e}^{t}}^{\wz{d}}k_{\wz{d}}=p_{\wz{d}^{t},\wz{f}}^{\wz{e}}k_{\wz{e}}$;

\item\label{jb-3}  $\sum_{\wz{h}\in\wz{\partial}(\Gamma)}p_{\wz{d},\wz{h}}^{\wz{f}}=k_{\wz{d}}$;

\item\label{jb-4} $\sum_{\wz{h}\in\wz{\partial}(\Gamma)}p_{\wz{d},\wz{e}}^{\wz{h}}p_{\wz{g},\wz{h}}^{\wz{f}}=\sum_{\wz{l}\in\wz{\partial}(\Gamma)}p_{\wz{g},\wz{d}}^{\wz{l}}p_{\wz{l},\wz{e}}^{\wz{f}}$.
\end{enumerate}
\end{lemma}

We recall the definitions of pure arcs and mixed arcs introduced in \cite{YYF20}. Let $T=\{q\mid(1,q-1)\in\wz{\partial}(\Gamma)\}$. For $q\in T$, an arc of type $(1,q-1)$ in $\Gamma$ is said to be \emph{pure}, if every circuit of length $q$ containing it consists of arcs of type $(1,q-1)$; otherwise, this arc is said to be \emph{mixed}. We say that $(1,q-1)$ is pure if any arc of type $(1,q-1)$ is pure; otherwise, we say that $(1,q-1)$ is mixed.

\begin{lemma}\label{not pure}
Let $p_{(1,s-1),(1,t-1)}^{(2,q-2)}\neq0$ with $q\geq3$. If $q\in T\backslash\{s\}$, then $(1,q-1)$ is mixed.
\end{lemma}
\begin{proof}
Suppose that $(1,q-1)$ is pure. Then there exists a circuit $(x_0,x_1,\ldots,x_{q-1})$ consisting of arcs of type $(1,q-1)$, where the indices are
read modulo $q$. It follows that $\partial(x_i,x_{i+j})=j$ for $1\leq j<q$. Since $(x_0,x_2)\in\Gamma_{2,q-2}$ and $p_{(1,s-1),(1,t-1)}^{(2,q-2)}\neq0$, there exists $x_1'\in P_{(1,s-1),(1,t-1)}(x_0,x_2)$. This implies that $(x_0,x_1',x_2,\ldots,x_{q-1})$ is a circuit of length $q$ containing an arc of type $(1,q-1)$ with $(x_0,x_1')\notin\Gamma_{1,q-1}$, a contradiction. Thus, $(1,q-1)$ is mixed.
\end{proof}

For two nonempty subsets $E$ and $F$ of $R$, define
\begin{eqnarray*}
EF&:=&\{\Gamma_{\tilde{h}}\mid\sum_{\Gamma_{\tilde{i}}\in E}\sum_{\Gamma_{\tilde{j}}\in F}p_{\tilde{i},\tilde{j}}^{\tilde{h}}\neq0\}.
\end{eqnarray*}
We write $\Gamma_{\tilde{i}}\Gamma_{\tilde{j}}$ instead of $\{\Gamma_{\tilde{i}}\}\{\Gamma_{\tilde{j}}\}$, and $\Gamma_{\tilde{i}}^2$ instead of $\{\Gamma_{\tilde{i}}\}\{\Gamma_{\tilde{i}}\}$.

\begin{lemma}\label{comm}
	If $\wz{h},\wz{i}\in \wz{\partial}(\Gamma)$ and $\wz{j}\in \{\wz{i}, \wz{i}^t\}$, then $\Gamma_{\tilde{h}}\Gamma_{\tilde{j}^t}=\{\Gamma_{\tilde{i}}\}$ if and only if $p_{\tilde{i},\tilde{j}}^{\tilde{h}}=k_{\wz{i}}$.
\end{lemma}
\begin{proof}
The sufficiency is immediate. We now prove the necessity. Assume $\Gamma_{\tilde{h}}\Gamma_{\tilde{j}^t}=\{\Gamma_{\tilde{i}}\}$. Since $k_{\wz{i}}=k_{\wz{j}}$, by setting $\wz{d}=\wz{h}$ and $\wz{e}=\wz{j}^t$ in Lemma \ref{jb} \ref{jb-1}, one has $p_{\wz{h},\wz{j}^t}^{\wz{i}}=k_{\wz{h}}$, which implies $p_{\tilde{i},\tilde{j}}^{\tilde{h}}=k_{\wz{i}}$ from Lemma \ref{jb} \ref{jb-2}. The desired result follows.
\end{proof}

For a nonempty subset $F$ of $R$, we say $F$ {\em closed} if $\Gamma_{\wz{i}^{t}}\Gamma_{\wz{j}}\subseteq F$ for any $\Gamma_{\wz{i}}$ and $\Gamma_{\wz{j}}$ in $F$. Let $\langle F\rangle$ be the minimum closed subset containing $F$. Denote $F(x)=\{y\in V(\Gamma)\mid(x,y)\in \cup_{f\in F}f\}$. For a subset $I$ of $T$ and a vertex $x\in V(\Gamma)$, let $\Delta_{I}(x)$ be the digraph with the vertex set $F_{I}(x)$ and arc set $(\cup_{q\in I}\Gamma_{1,q-1})\cap(F_I(x))^2$, where $F_{I}=\langle\{\Gamma_{1,q-1}\}_{q\in I}\rangle$. For the sake of convenience, we also write $F_q$ (resp. $\Delta_{q}(x)$) instead of $F_{\{q\}}$ (resp. $\Delta_{\{q\}}(x)$). If the digraph $\Delta_{I}(x)$ does not depend on the choice of vertex $x$ up to isomorphism and no confusion occurs, we write $\Delta_{I}$ instead of $\Delta_{I}(x)$.

\begin{lemma}\label{semicoplete}
	Let $I$ be a nonempty subset of $T$. Suppose that $\Delta_I(x)$ is semicomplete for some  $x\in V(\Gamma)$. Then $\Delta_I(y)$ is a weakly distance-regular digraph for all $y\in V(\Gamma)$. Moreover, $p_{\wz{i},\wz{j}}^{\wz{h}}(\Delta_I(x))=p_{\wz{i},\wz{j}}^{\wz{h}}(\Delta_I(y))$ for all $y\in V(\Gamma)$ and $\wz{i}$, $\wz{j}$, $\wz{h}$.
\end{lemma}
\begin{proof}
	Since $\Gamma$ is weakly distance-regular and $\Delta_I(x)$ is semicomplete, one obtains $F_I=\{\Gamma_{1,r-1},\Gamma_{r-1,1}\mid r\in I\}\cup\{\Gamma_{0,0}\}$.
	
	Let $y\in V(\Gamma)$. Pick distinct vertices $z,w\in F_I(y)$. Since $F_I=\{\Gamma_{1,r-1},\Gamma_{r-1,1}\mid r\in I\}\cup\{\Gamma_{0,0}\}$, we may assume that $\wz{\partial}_{\Delta_I(y)}(z,w)=(1,i)$ with $i\geq1$. Suppose $i>2$. For $u\in N_{\Delta_I(y)}^+(w)$, since $\partial_{\Delta_I(y)}(w,z)>2$ and $\Delta_I(y)$ is semicomplete, we have $(z,u)\in A(\Delta_I(y))$, and so $N_{\Delta_I(y)}^+(w)\cup\{w\}\subseteq N_{\Delta_I(y)}^+(z)$, contrary to the fact that $\Delta_I(y)$ is regular. Then $i\in\{1,2\}$. It follows that $\wz{\partial}_{\Gamma}(z,w)=\wz{\partial}_{\Delta_I(y)}(z,w)$, and so $I\subseteq\{2,3\}$. Since the distinct vertices $z,w\in F_I(y)$ were arbitrary, one gets $\wz{\partial}(\Delta_I(y))\subseteq\{(0,0),(1,1),(1,2),(2,1)\}$.
	It follows that
	\begin{align}
	[\Delta_I(y)]_{\wz{i}}(z)\cap[\Delta_I(y)]_{\wz{j}^t}(w)=P_{\wz{i},\wz{j}}(z,w)\label{gongshi2.4}
	\end{align}
     for $\wz{i},\wz{j}\in\wz{\partial}(\Delta_I(y))$ and $z,w\in F_I(y)$. By the weakly distance-regularity of $\Gamma$, $\Delta_I(y)$ is weakly distance-regular. Since $y\in V(\Gamma)$ was arbitrary, from \eqref{gongshi2.4}, $p_{\wz{i},\wz{j}}^{\wz{h}}(\Delta_I(x))=p_{\wz{i},\wz{j}}^{\wz{h}}(\Delta_I(y))$ for all $y\in V(\Gamma)$ and $\wz{i},\wz{j},\wz{h}$.
\end{proof}

\begin{lemma}\label{diameter 2}
	{\rm(\cite[Proposition 3.5]{YYF})} If $\Gamma$ is a semicomplete digraph, then $\wz{\partial}(\Gamma)\subseteq\{(0,0),(1,1),(1,2),(2,1)\}$.
\end{lemma}
In the remainder of this section, $\Gamma$ always denotes  a locally semicomplete commutative weakly distance-regular digraph. The commutativity of $\Gamma$ will be used frequently in the sequel, so we no longer refer to it for the sake of simplicity.

\begin{lemma}\label{bkb}
	  Let $q\in T$ and $I=\{r\mid\Gamma_{1,r-1}\in\Gamma_{1,q-1}\Gamma_{q-1,1}\}\neq\emptyset$. If $\Gamma_{1,r-1}\Gamma_{q-1,1}=\{\Gamma_{1,q-1}\}$ for all $r\in I$, then $\Delta_I(x)$ is a semicomplete weakly distance-regular digraph with $p_{\wz{i},\wz{j}}^{\wz{h}}(\Delta_I(x))=p_{\wz{i},\wz{j}}^{\wz{h}}(\Delta_I(y))$ for all $x,y\in V(\Gamma)$ and $\wz{i}$, $\wz{j}$, $\wz{h}$.
\end{lemma}
\begin{proof}
	By Lemma \ref{semicoplete}, it suffices to show that $\Delta_I(x)$ is semicomplete. Pick distinct vertices $y,z\in F_I(x)$. Then there exists a path $(y=x_0,x_1,\ldots,x_l=z)$ in $\Delta_I(x)$. Let $w\in\Gamma_{1,q-1}(y)$ and $(x_i,x_{i+1})\in\Gamma_{1,r_i-1}$ for $0\leq i\leq l-1$. Since $\Gamma_{1,r-1}\Gamma_{q-1,1}=\{\Gamma_{1,q-1}\}$ for all $r\in I$, one obtains $p_{(1,q-1),(q-1,1)}^{(1,r_i-1)}=k_{1,q-1}$ from Lemma \ref{comm}, which implies that $(x_i,w)\in\Gamma_{1,q-1}$ for $0\leq i\leq l$ by induction. Since $w\in P_{(1,q-1),(q-1,1)}(y,z)$ and $\Gamma$ is locally semicomplete, one gets $(y,z)\in\Gamma_{1,r-1}$ or $(z,y)\in\Gamma_{r-1,1}$ for some $r\in I$. Since the distinct vertices $y,z\in F_I(x)$ were arbitrary, $\Delta_I(x)$ is semicomplete.
\end{proof}

\begin{lemma}\label{1,1}
Let $2\in T$. If $\Gamma_{1,2}\notin\Gamma_{1,1}^2$, then $\Delta_{2}$ is isomorphic to $K_{k_{1,1}+1}$.
\end{lemma}
\begin{proof}
Let $x\in V(\Gamma)$. Pick distinct vertices $y,z\in F_{2}(x)$. Since $\Gamma_{1,2}\notin\Gamma_{1,1}^2$ and $\Gamma$ is locally semicomplete, one has $\Gamma_{1,1}^2\subseteq\{\Gamma_{0,0},\Gamma_{1,1}\}$. By induction, we get $\Gamma_{1,1}^i\subseteq\{\Gamma_{0,0},\Gamma_{1,1}\}$ for $i\geq2$, which implies $(y,z)\in\Gamma_{1,1}$. This implies $\Delta_2(x)\simeq K_{k_{1,1}+1}$.
\end{proof}

\begin{lemma}\label{(1,2) mixed}
Suppose that $(1,2)$ is mixed. Then $\Delta_{\{2,3\}}(x)$ is a semicomplete weakly distance-regular digraph with $p_{\wz{i},\wz{j}}^{\wz{h}}(\Delta_{\{2,3\}}(x))=p_{\wz{i},\wz{j}}^{\wz{h}}(\Delta_{\{2,3\}}(y))$ for all $x,y\in V(\Gamma)$ and $\wz{i},\wz{j},\wz{h}$. Moreover, $F_{\{2,3\}}=\{\Gamma_{0,0},\Gamma_{1,1},\Gamma_{1,2},\Gamma_{2,1}\}$.
\end{lemma}
\begin{proof}
Since $(1,2)$ is mixed, we have $\Gamma_{1,t-1}\in\Gamma_{1,q-1}^2$ with $\{q,t\}=\{2,3\}$.
Let $I=\{r\mid\Gamma_{1,r-1}\in F_q\}$. Pick a vertex $x\in V(\Gamma)$. Let $y,z$ be distinct vertices in $F_I(x)$. Since $F_q=F_I$, there exists a path $(y,x_1,\ldots,x_l=z)$ consisting of arcs of type $(1,q-1)$.

We show that $(y,z)$ or $(z,y)\in A(\Delta_{I}(x))$ by induction on $l$. Assume that $(y,x_{l-1})$ or $(x_{l-1},y)\in A(\Delta_{I}(x))$. It follows that there exists $r\in I$ such that $(y,x_{l-1})$ or $(x_{l-1},y)\in\Gamma_{1,r-1}$. Suppose $(x_{l-1},y)\in\Gamma_{1,r-1}$ or $q=2$. Then, $y,z\in N^+(x_{l-1})$ or $y,z\in N^-(x_{l-1})$. Since $\Gamma$ is locally semicomplete, we have $(y,z)$ or $(z,y)\in\Gamma_{1,s-1}$ for some $s\in I$, which implies that $(y,z)$ or $(z,y)\in A(\Delta_{I}(x))$. Suppose $(y,x_{l-1})\in\Gamma_{1,r-1}$ with $r>1$ and $q=3$. It follows that $\Gamma_{1,1}\in\Gamma_{1,2}^2$ and $2\in I$. By Lemma \ref{jb} \ref{jb-2}, we have $p_{(2,1),(1,1)}^{(1,2)}\neq0$, which implies that there exists $w\in P_{(2,1),(1,1)}(x_{l-1},z)$. The fact that $\Gamma$ is locally semicomplete implies $w=y$, $(y,w)\in\Gamma_{1,s-1}$ or $(w,y)\in\Gamma_{s-1,1}$ for some $s\in I$. Since $w=y$, $y,z\in N^-(w)$ or $y,z\in N^+(z)$, one gets $(y,z)$ or $(z,y)\in\Gamma_{1,h-1}$ for some $h\in I$, which implies $(y,z)$ or $(z,y)\in A(\Delta_I(x))$.

Since the distinct vertices $y,z\in F_{I}(x)$ were arbitrary, $\Delta_I(x)$ is semicomplete. By Lemma \ref{semicoplete}, the first statement is valid. The second statement is also valid from Lemma \ref{diameter 2}.
\end{proof}

\begin{lemma}\label{p(1,2),(2,1)^(1,1)=k_1,2}
If $(1,2)$ is pure and $2\in T$, then $\Gamma_{1,1}\Gamma_{1,2}=\{\Gamma_{1,2}\}$.
\end{lemma}
\begin{proof}
Let $(x,y)\in\Gamma_{1,1}$ and $(y,z)\in\Gamma_{1,2}$. Since $(1,2)$ is pure, we have $p_{(1,2),(1,2)}^{(2,1)}\neq0$, which implies that there exists $w\in P_{(1,2),(1,2)}(z,y)$. Since $x,z\in N^+(y)$ and $w,x\in N^-(y)$, one gets $(x,z),(w,x)\in A(\Gamma)$. It follows that $(z,w,x)$ is a circuit containing an arc of type $(1,2)$. Since $(1,2)$ is pure, one has $(x,z)\in\Gamma_{1,2}$. Since $z\in\Gamma_{1,2}(y)$ was arbitrary,  we obtain $\Gamma_{1,1}\Gamma_{1,2}=\{\Gamma_{1,2}\}$.
\end{proof}

\begin{lemma}\label{2,2}
Let $\Gamma_{2,i}\in\Gamma_{1,2}\Gamma_{1,q}$ with $i>1$ and $q>0$. If $(1,2)$ is pure, then $i=2$ and $q\in\{2,3\}$.
\end{lemma}
\begin{proof}
	Let $(x,z)\in\Gamma_{2,i}$ and $y\in P_{(1,2),(1,q)}(x,z)$. Since $\Gamma$ is locally semicomplete and $i>1$, we have $q>1$. The fact that $(1,2)$ is pure implies $p_{(1,2),(1,2)}^{(2,1)}\neq0$. It follows that there exists $y'\in P_{(1,2),(1,2)}(y,x)$. Since $z,y'\in N^+(y)$ and $\Gamma$ is locally semicomplete, one gets $(y',z)$ or $(z,y')\in A(\Gamma)$. If $(y',z)\in A(\Gamma)$, then $(x,z)$ or $(z,x)\in A(\Gamma)$ since $x,z\in N^+(y')$, contrary to the fact that $(x,z)\in\Gamma_{2,i}$ with $i>1$. Hence, $(z,y')\in A(\Gamma)$. Since $(z,y',x)$ is a path, we get $\partial(z,x)=i=2$. The fact $q=\partial(z,y)\leq1+\partial(y',y)=3$ implies $q\in\{2,3\}$.
\end{proof}

\section{Arcs of type $(1,3)$}
In this section, we always assume that $\Gamma$ is a locally semicomplete commutative weakly distance-regular digraph. For $q\in T$, we say that the configuration C$(q)$ (resp. D$(q)$) exists if $p_{(1,q-1),(1,q-1)}^{(1,q-2)}\neq0$ (resp. $p_{(1,q-2),(q-2,1)}^{(1,q-1)}\neq0$) and $(1,q-2)$ is pure. The main result of this section is as follows, which characterizes mixed arcs of type $(1,3)$.

\begin{prop}\label{C4 holds}
	If $(1,3)$ is mixed, then the following hold:
	\begin{enumerate}
		\item\label{C4-3} {\rm C}$(4)$ exists;
	
		\item\label{C4-2} $\Gamma_{1,2}\Gamma_{3,1}=\{\Gamma_{1,3}\}$;

		\item\label{C4-1} $\Gamma_{1,3}^2=\{\Gamma_{1,2}\}$.
	\end{enumerate}
\end{prop}

In order to prove Proposition \ref{C4 holds}, we need some auxiliary lemmas.

\begin{lemma}\label{1,3}
	Suppose that $(1,3)$ is mixed. Then {\rm C}$(4)$ or {\rm D}$(4)$ exists.
\end{lemma}
\begin{proof}
	Let $(x,y)$ be an arc of type $(1,3)$. Suppose that each shortest path from $y$ to $x$ does not contain an arc of type $(1,2)$. Since $(1,3)$ is mixed, there exists $z\in\Gamma_{1,1}(y)$ such that $\partial(z,x)=2$. The fact that $\Gamma$ is locally semicomplete implies that $(x,z)\in\Gamma_{1,2}$. Since each shortest path from $y$ to $x$ does not contain an arc of type $(1,2)$, each shortest path from $z$ to $x$ consisting of edges, which implies $\Gamma_{1,2}\in\Gamma_{1,1}^2$. Then $(1,2)$ is mixed. By Lemma \ref{(1,2) mixed}, one gets $F_{\{2,3\}}=\{\Gamma_{0,0},\Gamma_{1,2},\Gamma_{2,1},\Gamma_{1,1}\}$, contrary to the fact $z\in P_{(1,2),(1,1)}(x,y)$. Then there exists $z\in\Gamma_{1,2}(y)$ such that $\partial(z,x)=2$.
	
	Assume the contrary, namely, $(1,2)$ is mixed. Then $\Gamma_{2,1}\in\Gamma_{1,1}^2\cup\Gamma_{1,1}\Gamma_{1,2}$. Pick a vertex $w\in P_{(1,1),(1,i)}(z,y)$ for some $i\in\{1,2\}$. Since $x,w\in N^-(y)$, we get $(x,w)$ or $(w,x)\in A(\Gamma)$. The fact $x,z\in N^-(w)$ or $x,z\in N^+(w)$ implies $(x,z)\in\Gamma_{1,2}$. Since $z\in P_{(1,2),(2,1)}(x,y)$, one has $\Gamma_{1,3}\in F_{\{2,3\}}$, contrary to Lemma \ref{(1,2) mixed}. Thus, $(1,2)$ is pure.
	
	Suppose that D$(4)$ does not exist. It suffices to show that $p_{(1,3),(1,3)}^{(1,2)}\neq0$. Since D$(4)$ does not exist, we have $\Gamma_{1,3}\notin\Gamma_{1,2}\Gamma_{2,1}$. It follows that $z\notin P_{(1,2),(2,1)}(x,y)$, and so $(x,z)\in\Gamma_{2,2}$. Since $(1,2)$ is pure, there exists $w\in P_{(1,2),(1,2)}(z,y)$. Since $(x,z)\in\Gamma_{2,2}$ and $x,w\in N^-(y)$, we have $(w,x)\in\Gamma_{1,s}$ with $s\in\{2,3\}$. Since $\Gamma_{1,3}\notin\Gamma_{1,2}\Gamma_{2,1}$, we have $w\notin P_{(2,1),(1,2)}(x,y)$, and so $s=3$. The fact $x\in P_{(1,3),(1,3)}(w,y)$ implies $p_{(1,3),(1,3)}^{(1,2)}\neq0$. The desired result follows.
\end{proof}

\begin{lemma}\label{2,2-1,3}
	Let $i,q$ be positive integers and $\Gamma_{2,i}\in\Gamma_{1,3}\Gamma_{1,q}$. If $(1,3)$ is mixed, then $i=2$ and $q\in\{2,3\}$.
\end{lemma}
\begin{proof}
	Let $(x,z)\in\Gamma_{2,i}$ and $y\in P_{(1,3),(1,q)}(x,z)$. It follows that $i>1$. Since $(x,z)\in\Gamma_{2,i}$ and $\Gamma$ is locally semicomplete, we have $q>1$. Since $(1,3)$ is mixed, from Lemma \ref{1,3}, we have $p_{(1,3),(1,3)}^{(1,2)}\neq0$ or $p_{(1,2),(2,1)}^{(1,3)}\neq0$. By Lemma \ref{jb} \ref{jb-2}, we have $p_{(1,2),(3,1)}^{(1,3)}\neq0$ or $p_{(1,2),(2,1)}^{(1,3)}\neq0$. It follows that there exists $w\in\Gamma_{1,2}(x)$ such that $(y,w)\in\Gamma_{1,2}\cup\Gamma_{1,3}$. Since $w,z\in N^+(y)$ and $(x,z)\in\Gamma_{2,i}$ with $i>1$, we have $(w,z)\in\Gamma_{1,r}$ with $r\geq1$. Since $w\in P_{(1,2),(1,r)}(x,z)$, from Lemma \ref{2,2}, one obtains $i=2$ and $q\in\{2,3\}$.
\end{proof}

\begin{lemma}\label{semi}
	Let $(1,2)$ be pure and $I=\{r\mid\Gamma_{1,r-1}\in F_3\}$.
	\begin{enumerate}
		\item\label{semi-2} If $(1,3)$ is mixed, then $\Gamma_{2,2}\in\Gamma_{1,2}\Gamma_{1,3}\cup\Gamma_{1,3}^2$;
		
			\item\label{semi-1}  If $\Gamma_{2,2}\notin\Gamma_{1,2}^2\cup\Gamma_{1,2}\Gamma_{1,3}$, then $\Delta_I(x)$ is semicomplete for each $x\in V(\Gamma)$.
	\end{enumerate}
\end{lemma}
\begin{proof}
	Suppose that $(1,3)$ is mixed and $\Gamma_{2,2}\notin\Gamma_{1,2}\Gamma_{1,3}\cup\Gamma_{1,3}^2$, or $\Gamma_{2,2}\notin\Gamma_{1,2}^2\cup\Gamma_{1,2}\Gamma_{1,3}$. Let $q=4$ if $(1,3)$ is mixed and $\Gamma_{2,2}\notin\Gamma_{1,2}\Gamma_{1,3}\cup\Gamma_{1,3}^2$, and $q=3$ if $\Gamma_{2,2}\notin\Gamma_{1,2}^2\cup\Gamma_{1,2}\Gamma_{1,3}$.

We claim that $J:=\{r\mid\Gamma_{1,r-1}\in F_q\}\subseteq\{2,3\}$ and $\Delta_J(x)$ is semicomplete for all $x\in V(\Gamma)$. Let $x\in V(\Gamma)$. Pick distinct vertices $y,z\in F_J(x)$. Since $F_J(x)=F_q(x)$, there exists a path $(x_0=y,x_1,\ldots,x_l=z)$ consisting of arcs of type $(1,q-1)$. We prove $(y,z)$ or $(z,y)\in A(\Delta_J(x))$ by induction on $l$. Now assume $(y,x_{l-1})$ or $(x_{l-1},y)\in A(\Delta_J(x))$. Then $(x_{l-1},y)\in\Gamma_{1,r}$ or $(y,x_{l-1})\in\Gamma_{1,r}$ for some $\Gamma_{1,r}\in F_q$. If $(x_{l-1},y)\in\Gamma_{1,r}$ for some $\Gamma_{1,r}\in F_q$, then $(y,z)$ or $(z,y)\in\Gamma_{1,s}\in F_q$ since $y,z\in N^+(x_{l-1})$. We only need to consider the case $(y,x_{l-1})\in\Gamma_{1,r}\in F_q$ with $r>1$. If $(y,z)\in\Gamma_{2,i}$ with $i>1$, from Lemmas \ref{2,2} and \ref{2,2-1,3}, then $i=2$ and $r\in\{2,3\}$ since $x_{l-1}\in P_{(1,r),(1,q-1)}(y,z)$ with $q\in\{3,4\}$, contrary to the fact that $\Gamma_{2,2}\notin\Gamma_{1,q-1}^2\cup\Gamma_{1,2}\Gamma_{1,3}$. Hence, $(y,z)\in\Gamma_{1,s}\in\Gamma_{1,r}\Gamma_{1,q-1}$ or $(y,z)\in\Gamma_{s,1}\in\Gamma_{1,r}\Gamma_{1,q-1}$. Since $\Gamma_{1,r}\in F_q$, we obtain $\Gamma_{1,s}\in F_q$, and so $(y,z)$ or $(z,y)\in A(\Delta_J(x))$. Since $y,z\in F_J(x)$ were arbitrary, $\Delta_J(x)$ is semicomplete. By Lemmas \ref{semicoplete} and \ref{diameter 2}, we have $J\subseteq\{2,3\}$. Thus, our claim is valid.
	
(i) is immediate from the claim.
		
(ii) Now we have $q=3$. By the claim, we have $I=J\subseteq\{2,3\}$, and $\Delta_I(x)$ is semicomplete. Thus, \ref{semi-1} is valid.
\end{proof}

\begin{lemma}\label{1,2^2}
	Suppose that $(1,2)$ is pure. Then the following hold:
	\begin{enumerate}
		\item\label{1,2^2,1} $\Gamma_{1,2}^2\subseteq\{\Gamma_{1,2},\Gamma_{1,3},\Gamma_{2,1},\Gamma_{2,2}\}$;
		
		\item\label{1,2 2,1} $\Gamma_{1,2}\Gamma_{2,1}\subseteq\{\Gamma_{0,0},\Gamma_{1,1},\Gamma_{1,2},\Gamma_{1,3},\Gamma_{2,1},\Gamma_{3,1}\}$.
	\end{enumerate}
\end{lemma}
\begin{proof}
	(i) Since $(1,2)$ is pure, we have $\Gamma_{1,2}^2\subseteq\{\Gamma_{1,2},\Gamma_{1,3},\Gamma_{1,4},\Gamma_{2,1},\Gamma_{2,2}\}$ from Lemma \ref{2,2}. It suffices to show that $\Gamma_{1,4}\notin\Gamma_{1,2}^2$. Assume the contrary, namely, $\Gamma_{1,4}\in\Gamma_{1,2}^2$. Let $I=\{r\mid\Gamma_{1,r-1}\in F_3\}$. Since $5\in I$, from Lemmas \ref{semicoplete} and \ref{diameter 2}, $\Delta_I(x)$ is not semicomplete for $x\in V(\Gamma)$. Lemma \ref{semi} \ref{semi-1} implies $\Gamma_{2,2}\in\Gamma_{1,2}\Gamma_{1,q}$ for some $q\in\{2,3\}$.
	
Let $(x,z)\in\Gamma_{2,2}$ and $y\in P_{(1,2),(1,q)}(x,z)$. Since $p_{(1,2),(1,2)}^{(1,4)}\neq0$, from Lemma \ref{jb} \ref{jb-2}, we have $p_{(1,4),(2,1)}^{(1,2)}\neq0$. Pick a vertex $w\in P_{(1,4),(2,1)}(x,y)$. Since $(x,z)\in\Gamma_{2,2}$ and $z,w\in N^+(y)$, one gets $(w,z)\in A(\Gamma)$. It follows that $\partial(w,x)\leq1+\partial(z,x)=3$, a contradiction. Thus, \ref{1,2^2,1} is valid.
	
	(ii) Let $x,y,z$ be vertices such that $(x,y),(z,y)\in\Gamma_{1,2}$. Note that $\partial(z,x)\leq 1+\partial(y,x)=3$. Since $\Gamma$ is locally semicomplete, \ref{1,2 2,1} is valid.
\end{proof}

\begin{lemma}\label{1,3 3,1}
	Suppose that $(1,3)$ is mixed. Then the following hold:
	\begin{enumerate}
		\item\label{1,2 1,3}
		$\Gamma_{1,2}\Gamma_{1,3}\cup\Gamma_{1,3}^2\subseteq\{\Gamma_{1,2},\Gamma_{1,3},\Gamma_{2,2}\}$;

		\item\label{1,3 3,1-3} $\Gamma_{1,3}\Gamma_{3,1}\subseteq\{\Gamma_{0,0},\Gamma_{1,1},\Gamma_{1,2},\Gamma_{2,1},\Gamma_{1,3},\Gamma_{3,1}\}$;
		
		\item\label{1,3 3,1-1} $\Gamma_{1,3}\Gamma_{2,1}\subseteq\{\Gamma_{1,2},\Gamma_{2,1},\Gamma_{1,3},\Gamma_{3,1}\}$;
		
		\item\label{1,3 3,1-2} $\Gamma_{1,2}\Gamma_{3,1}\subseteq\{\Gamma_{1,2},\Gamma_{2,1},\Gamma_{1,3},\Gamma_{3,1}\}$.
	
	\end{enumerate}
\end{lemma}
\begin{proof}
	We claim that $\Gamma_{1,4}\notin\Gamma_{1,3}\Gamma_{3,1}\cup\Gamma_{1,3}\Gamma_{2,1}$.
	Assume the contrary, namely, $\Gamma_{1,4}\in\Gamma_{1,3}\Gamma_{q,1}$ for some $q\in\{2,3\}$. Let $(x,z)\in\Gamma_{1,4}$ and $y\in P_{(1,3),(q,1)}(x,z)$. Lemma \ref{1,3} implies $p_{(1,3),(1,3)}^{(1,2)}\neq0$ or $p_{(1,2),(2,1)}^{(1,3)}\neq0$. By Lemma \ref{jb} \ref{jb-2}, one gets $p_{(1,2),(3,1)}^{(1,3)}\neq0$ or $p_{(1,2),(2,1)}^{(1,3)}\neq0$, which implies that there exists $w\in P_{(1,2),(r,1)}(x,y)$ for some $r\in\{2,3\}$. The fact $4=\partial(z,x)\leq\partial(z,w)+\partial(w,x)$ implies $(z,w)\notin A(\Gamma)$. Since $w,z\in N^+(x)$, we get $(w,z)\in A(\Gamma)$. Since $(w,z,y)$ is a circuit, we obtain $q=r=2$. By Lemma \ref{1,3} again, $(1,2)$ is pure, which implies $(w,z)\in\Gamma_{1,2}$. Since $w\in P_{(1,2),(1,2)}(x,z)$, one has $\Gamma_{1,4}\in\Gamma_{1,2}^2$, contrary to Lemma \ref{1,2^2} \ref{1,2^2,1}. Thus, the claim is valid.
	
	(i) Let $\Gamma_{1,q}\in\Gamma_{1,3}\Gamma_{1,r}$ with $q\geq2$ and $r\in\{2,3\}$. Since $(1,2)$ is pure from Lemma \ref{1,3}, by Lemma \ref{semi} \ref{semi-2}, there exist vertices $x,y,z$ such that $(x,z)\in\Gamma_{2,2}$ and $y\in P_{(1,3),(1,s)}(x,z)$ with $s\in\{2,3\}$. Since $p_{(1,3),(1,r)}^{(1,q)}\neq0$, from Lemma \ref{jb} \ref{jb-2}, one obtains $p_{(1,q),(r,1)}^{(1,3)}\neq0$, which implies that there exists $w\in P_{(1,q),(r,1)}(x,y)$. Since $w,z\in N^+(y)$ and $(x,z)\in\Gamma_{2,2}$, we have $(w,z)\in A(\Gamma)$. Then $q=\partial(w,x)\leq 1+\partial(z,x)=3$. By Lemma \ref{2,2-1,3}, we get $\Gamma_{1,2}\Gamma_{1,3}\cup\Gamma_{1,3}^2\subseteq\{\Gamma_{1,2},\Gamma_{1,3},\Gamma_{2,2}\}$. Thus, \ref{1,2 1,3} is valid.
	
	(ii) Since $\Gamma$ is locally semicomplete, from the claim, \ref{1,3 3,1-3} is valid.
	
	(iii) By Lemma \ref{1,3}, $(1,2)$ is pure. It follows from Lemma \ref{p(1,2),(2,1)^(1,1)=k_1,2} that $p_{(1,1),(1,2)}^{(1,3)}=0$. In view of Lemma \ref{jb} \ref{jb-2}, one has $p_{(1,3),(2,1)}^{(1,1)}=0$. Since $\Gamma$ is locally semicomplete, from the claim, \ref{1,3 3,1-1} is valid.
	
	(iv)  is also valid from \ref{1,3 3,1-1}.
\end{proof}

\begin{lemma}\label{2,2 1}
	Let $(1,2)$ be pure and $I=\{r\mid\Gamma_{1,r-1}\in F_3\}$.
	\begin{enumerate}
		\item\label{2,2 1-2} If $(1,3)$ is mixed, then $\Gamma_{2,2}\Gamma_{1,3}\subseteq\{\Gamma_{2,1},\Gamma_{3,1},\Gamma_{2,2}\}$;
		
		\item\label{2,2 1-1}  If $\Delta_I(x)$ is not semicomplete for some $x\in V(\Gamma)$, 
		then $\Gamma_{2,2}\Gamma_{1,2}\subseteq\{\Gamma_{2,1},\Gamma_{3,1},\Gamma_{2,2}\}$.
	\end{enumerate}
\end{lemma}
\begin{proof}
		Suppose that $(1,3)$ is mixed, or $\Delta_I(x)$ is not semicomplete for some $x\in V(\Gamma)$. Let $q=3$ if $(1,3)$ is mixed, and $q=2$ if $\Delta_I(x)$ is not semicomplete.

We claim that $\Gamma_{2,2}\Gamma_{1,q}\subseteq\{\Gamma_{2,1},\Gamma_{3,1},\Gamma_{2,2}\}$. By Lemma \ref{semi},
	we have $p_{(1,q),(1,r)}^{(2,2)}\neq0$ for some $r\in\{2,3\}$. Let $(x,z)\in\Gamma_{2,2}$, $y\in P_{(1,q),(1,r)}(x,z)$ and $(x,w)\in\Gamma_{1,q}$. Since $(z,x)\in\Gamma_{2,2}$ and $\Gamma$ is locally semicomplete, we get $(z,w)\notin A(\Gamma)$. Since $x\in P_{(q,1),(1,q)}(w,y)$, from Lemma \ref{1,2^2} \ref{1,2 2,1} and Lemma \ref{1,3 3,1} \ref{1,3 3,1-3}, one gets $(w,y)\in\Gamma_{0,0}$, $(w,y)\in\Gamma_{1,s}$ or $(w,y)\in\Gamma_{s,1}$ for some $s\in\{1,2,3\}$. If $(w,y)\in\Gamma_{0,0}$, then $w=y$, and so $(z,w)\in\Gamma_{2,1}\cup\Gamma_{3,1}$. If $(w,y)\in\Gamma_{1,1}$, then $(z,w)\in\Gamma_{2,1}\cup\Gamma_{3,1}$ since $z,w\in N^+(y)$. If $(w,y)\in\Gamma_{s,1}$ for some $s\in\{2,3\}$, from Lemma \ref{1,2^2} \ref{1,2 2,1} and Lemma \ref{1,3 3,1} \ref{1,3 3,1-3}--\ref{1,3 3,1-2}, then $(z,w)\in\Gamma_{2,1}\cup\Gamma_{3,1}$ since $y\in P_{(s,1),(1,r)}(w,z)$. If $(w,y)\in\Gamma_{1,s}$ for some $s\in\{2,3\}$, by Lemma \ref{1,2^2} \ref{1,2^2,1} and Lemma \ref{1,3 3,1} \ref{1,2 1,3}, then $(z,w)\in\Gamma_{2,2}\cup\Gamma_{2,1}\cup\Gamma_{3,1}$ since $y\in P_{(1,s),(1,r)}(w,z)$. Thus, the claim is valid.

(i) Since $(1,3)$ is mixed, from the claim, \ref{2,2 1-2} is valid.
	
(ii) Since $\Delta_I(x)$ is not semicomplete, from the claim, \ref{2,2 1-1} is also valid.
\end{proof}

%
%

\begin{lemma}\label{jb2}
	Suppose that $(1,3)$ is mixed. The following hold:
	\begin{enumerate}
		\item\label{jb2-1} $p_{(1,2),(1,2)}^{(1,3)}=0$;
		
		\item\label{jb2-2} $p_{(1,3),(1,2)}^{(1,3)}p_{(1,3),(1,2)}^{(1,2)}=0$;
		
		\item\label{jb2-3} $p_{(1,3),(1,2)}^{(2,2)}p_{(2,1),(2,2)}^{(2,2)}=p_{(1,3),(1,2)}^{(2,2)}p_{(2,1),(1,3)}^{(1,3)}$.
	\end{enumerate}
\end{lemma}
\begin{proof}
	Since $\Gamma$ is locally semicomplete, we have $p_{(1,r),(s,1)}^{(2,2)}=0$ for $r,s\in\{2,3\}$. By Lemma \ref{jb} \ref{jb-2}, one gets $p_{(1,s),(2,2)}^{(1,r)}=0$.
	
	 By setting $\wz{d}=\wz{e}=\wz{f}=(1,3)$ and $\wz{g}=(1,2)$ in Lemma \ref{jb} \ref{jb-4}, from Lemma \ref{1,3 3,1} \ref{1,2 1,3}, we have
	$$p_{(1,3),(1,3)}^{(1,2)}p_{(1,2),(1,2)}^{(1,3)}+p_{(1,3),(1,3)}^{(1,3)}p_{(1,2),(1,3)}^{(1,3)}=p_{(1,2),(1,3)}^{(1,2)}p_{(1,2),(1,3)}^{(1,3)}+p_{(1,2),(1,3)}^{(1,3)}p_{(1,3),(1,3)}^{(1,3)},$$ which implies
	\begin{align}
	p_{(1,3),(1,3)}^{(1,2)}p_{(1,2),(1,2)}^{(1,3)}=p_{(1,2),(1,3)}^{(1,2)}p_{(1,2),(1,3)}^{(1,3)}.\label{p2,2-1}
	\end{align}
	
(i) Since $p_{(1,s),(2,2)}^{(1,r)}=0$ for $r,s\in\{2,3\}$, by setting $\wz{d}=\wz{e}=\wz{f}=(1,2)$ and $\wz{g}=(1,3)$ in Lemma \ref{jb} \ref{jb-4}, one has
	\begin{align}
	p_{(1,2),(1,2)}^{(1,3)}p_{(1,3),(1,3)}^{(1,2)}+p_{(1,2),(1,2)}^{(2,1)}p_{(1,3),(2,1)}^{(1,2)}=p_{(1,3),(1,2)}^{(1,3)}p_{(1,3),(1,2)}^{(1,2)}\nonumber
	\end{align}
	from Lemma \ref{1,2^2} \ref{1,2^2,1} and Lemma \ref{1,3 3,1} \ref{1,2 1,3}. It follows from \eqref{p2,2-1} that $p_{(1,2),(1,2)}^{(2,1)}p_{(1,3),(2,1)}^{(1,2)}=0$. Since $(1,2)$ is pure from Lemma \ref{1,3}, one gets $p_{(1,2),(1,2)}^{(2,1)}\neq0$, which implies $p_{(1,3),(2,1)}^{(1,2)}=0$. By Lemma \ref{jb} \ref{jb-2}, we obtain $p_{(1,2),(1,2)}^{(1,3)}=0$.
	
	(ii) is also valid from (i) and \eqref{p2,2-1}.
	
	(iii) By (i) and Lemma \ref{jb} \ref{jb-2}, we have $p_{(2,1),(1,3)}^{(1,2)}=0$. Note that $p_{(1,r),(s,1)}^{(2,2)}=0$ for all $r,s\in\{2,3\}$. By setting $\wz{d}=(1,3)$, $\wz{f}=(2,2)$ and $\wz{e}=\wz{g}^t=(1,2)$ in Lemma \ref{jb} \ref{jb-4}, we have $p_{(1,3),(1,2)}^{(2,2)}p_{(2,1),(2,2)}^{(2,2)}=p_{(2,1),(1,3)}^{(1,3)}p_{(1,3),(1,2)}^{(2,2)}$ from Lemma \ref{1,3 3,1} \ref{1,2 1,3} and \ref{1,3 3,1-1}. Thus, \ref{jb2-3} is valid.
\end{proof}

\begin{lemma}\label{D4 not hold}
	If $(1,3)$ is mixed, then $p_{(1,2),(2,1)}^{(1,3)}=0$ and {\rm C}$(4)$ exists.
\end{lemma}
\begin{proof}
	By Lemma \ref{1,3}, C$(4)$ or D$(4)$ exists, which implies that $(1,2)$ is pure. It follows that $p_{(1,2),(1,2)}^{(2,1)}\neq0$. Assume the contrary, namely, D$(4)$ exists. Then $p_{(1,2),(2,1)}^{(1,3)}\neq0$. By Lemma \ref{jb} \ref{jb-2}, we have $p_{(1,3),(1,2)}^{(1,2)}\neq0$. Lemma \ref{jb2} \ref{jb2-2} implies
\begin{align}\label{666}
p_{(1,3),(3,1)}^{(1,2)}=p_{(1,2),(1,3)}^{(1,3)}=p_{(2,1),(1,3)}^{(1,3)}=0.
\end{align}
In view of Lemma \ref{jb2} \ref{jb2-1}, one gets $p_{(2,1),(1,3)}^{(1,2)}=p_{(1,2),(1,2)}^{(1,3)}=0$. By setting $\wz{d}=\wz{e}=(1,3)$ and $\wz{f}=\wz{g}^t=(1,2)$ in Lemma \ref{jb} \ref{jb-4}, we obtain
	\begin{align}
	p_{(1,3),(1,3)}^{(1,2)}p_{(2,1),(1,2)}^{(1,2)}+p_{(1,3),(1,3)}^{(2,2)}p_{(2,1),(2,2)}^{(1,2)}=0\label{d4-1}
	\end{align}
	from Lemma \ref{1,3 3,1} \ref{1,2 1,3} and \ref{1,3 3,1-1}.

Let $I=\{r\mid\Gamma_{1,r-1}\in F_3\}$.	Since $\Gamma_{1,3}\in F_3$, from Lemmas \ref{semicoplete} and \ref{diameter 2}, $\Delta_I(x)$ is not semicomplete for all $x\in V\Gamma$.

	\textbf{Case 1.} $p_{(1,2),(1,2)}^{(2,2)}=0$.
	
Since $\Delta_I(x)$ is not semicomplete for all $x\in V\Gamma$, from Lemma \ref{semi} \ref{semi-1}, we have $p_{(1,2),(1,3)}^{(2,2)}\neq0$. In view of Lemma \ref{jb2} \ref{jb2-3}, one gets $p_{(2,2),(1,2)}^{(2,2)}=p_{(1,3),(2,1)}^{(1,3)}=0$ from \eqref{666}.
	
	Suppose $p_{(2,1),(1,2)}^{(1,2)}=0$. By Lemma \ref{jb} \ref{jb-2}, we have $\Gamma_{1,2}\notin\Gamma_{1,2}^2$. Since $p_{(1,2),(1,2)}^{(2,2)}=p_{(1,2),(1,2)}^{(1,3)}=0$, from Lemma \ref{1,2^2} \ref{1,2^2,1}, one gets $\Gamma_{1,2}^2=\{\Gamma_{2,1}\}$. 
	Let $(x,z)\in\Gamma_{2,2}$. Since $p_{(1,2),(1,3)}^{(2,2)}\neq0$, there exists $y\in P_{(1,2),(1,3)}(x,z)$. Pick a vertex $w\in\Gamma_{1,2}(y)$. The fact $\Gamma_{1,2}^2=\{\Gamma_{2,1}\}$ implies $(w,x)\in\Gamma_{1,2}$. Since $w,z\in N^+(y)$ and $(x,z)\in\Gamma_{2,2}$, we have $(z,w)\in A(\Gamma)$. Since $p_{(1,2),(1,2)}^{(2,2)}=0$ and $y\in P_{(2,1),(1,3)}(w,z)$, from Lemma \ref{1,3 3,1} \ref{1,3 3,1-1}, one has $w\in P_{(1,2),(3,1)}(y,z)$. Since $w\in\Gamma_{1,2}(y)$ was arbitrary, we get $p_{(1,2),(3,1)}^{(1,3)}=k_{1,2}$. It follows from Lemma \ref{jb} \ref{jb-3} that $p_{(1,2),(2,1)}^{(1,3)}=0$, a contradiction. Then $p_{(2,1),(1,2)}^{(1,2)}\neq0$.
	
	By \eqref{d4-1}, $p_{(1,3),(1,3)}^{(1,2)}=0$. In view of Lemma \ref{jb} \ref{jb-2}, we have $p_{(3,1),(1,2)}^{(1,3)}=0$. Since $\Gamma$ is locally semicomplete, we have $p_{(1,r),(s,1)}^{(2,2)}=0$ for $r,s\in\{2,3\}$. Since $p_{(1,2),(1,2)}^{(1,3)}=0$, by setting $\wz{d}=\wz{e}=(1,2)$, $\wz{g}=(3,1)$ and $\wz{f}=(2,2)$ in Lemma \ref{jb} \ref{jb-4}, from Lemma \ref{1,2^2} \ref{1,2^2,1} and Lemma \ref{1,3 3,1} \ref{1,3 3,1-2}, one obtains $$p_{(1,2),(1,2)}^{(2,1)}p_{(3,1),(2,1)}^{(2,2)}+p_{(1,2),(1,2)}^{(2,2)}p_{(3,1),(2,2)}^{(2,2)}=0,$$ contrary to the fact that $p_{(1,2),(1,2)}^{(2,1)}p_{(3,1),(2,1)}^{(2,2)}\neq0$.

	\textbf{Case 2.} $p_{(1,2),(1,2)}^{(2,2)}\neq0$.
	
	By Lemma \ref{jb} \ref{jb-2}, we have $p_{(2,1),(2,2)}^{(1,2)}\neq0$, which implies $p_{(1,3),(1,3)}^{(2,2)}=0$ from \eqref{d4-1}. In view of Lemma \ref{semi} \ref{semi-2}, one gets $p_{(1,2),(1,3)}^{(2,2)}\neq0$. Lemma \ref{jb2} \ref{jb2-3} and \eqref{666} implies $$p_{(2,2),(1,2)}^{(2,2)}=p_{(1,3),(2,1)}^{(1,3)}=p_{(1,3),(1,2)}^{(1,3)}=0.$$
	
	Suppose $p_{(1,3),(1,3)}^{(1,2)}=0$. Since $p_{(1,3),(1,3)}^{(2,2)}=0$, from Lemma \ref{1,3 3,1} \ref{1,2 1,3}, we have $\Gamma_{1,3}^2=\{\Gamma_{1,3}\}$. Let $(u,v)\in\Gamma_{1,3}$. It follows that $\Gamma_{1,3}(v)\cup\{v\}\subseteq\Gamma_{1,3}(u)$, a contradiction. Thus, $p_{(1,3),(1,3)}^{(1,2)}\neq0$. By \eqref{d4-1}, one gets $p_{(1,2),(2,1)}^{(1,2)}=0$.
	
	Let $(x,z)\in\Gamma_{2,2}$. The fact $p_{(1,3),(1,2)}^{(2,2)}\neq0$ implies that there exists $y\in P_{(1,3),(1,2)}(x,z)$. Pick a vertex $w\in P_{(1,2),(2,1)}(x,y)$. Since $(x,z)\in\Gamma_{2,2}$ and $(x,w)\in A(\Gamma)$, one has $(z,w)\notin A(\Gamma)$. Since $p_{(1,2),(2,1)}^{(1,2)}=0$ and $y\in P_{(2,1),(1,2)}(w,z)$, from Lemma \ref{1,2^2} \ref{1,2 2,1}, we have $w\in P_{(1,2),(1,3)}(y,z)$, and so $p_{(1,2),(2,1)}^{(1,3)}\leq p_{(1,2),(1,3)}^{(1,2)}$. By Lemma \ref{jb} \ref{jb-2}, one gets $p_{(1,2),(2,1)}^{(1,3)}k_{1,3}=p_{(1,2),(1,3)}^{(1,2)}k_{1,2}$, which implies $k_{1,3}\geq k_{1,2}$. Pick a vertex $w'\in P_{(1,3),(1,3)}(y,z)$. Since $p_{(1,3),(1,3)}^{(2,2)}=0$ and $y\in P_{(1,3),(1,3)}(x,w')$, from Lemma \ref{1,3 3,1} \ref{1,2 1,3}, one has $w'\in P_{(1,2),(3,1)}(x,y)$, and so $p_{(1,3),(1,3)}^{(1,2)}\leq p_{(1,2),(3,1)}^{(1,3)}$. By Lemma \ref{jb} \ref{jb-2}, we get $p_{(1,3),(1,3)}^{(1,2)}k_{1,2}=p_{(1,2),(3,1)}^{(1,3)}k_{1,3}$, which implies $k_{1,2}\geq k_{1,3}$. Thus, $k_{1,2}=k_{1,3}$.
	
	Since $p_{(1,2),(1,2)}^{(1,3)}=p_{(1,2),(1,2)}^{(1,2)}=0$ from Lemma \ref{jb} \ref{jb-2}, we have $\Gamma_{1,2}^2\subseteq\{\Gamma_{2,2},\Gamma_{2,1}\}$ by Lemma \ref{1,2^2} \ref{1,2^2,1}. Since $p_{(1,2),(2,2)}^{(2,2)}=p_{(1,3),(1,2)}^{(1,3)}=0$, by setting $\wz{d}=\wz{e}=(1,2)$, $\wz{g}=(1,3)$ and $\wz{h}=(2,2)$ in Lemma \ref{jb} \ref{jb-4}, from Lemma \ref{1,3 3,1} \ref{1,2 1,3} and \ref{1,3 3,1-1}, one gets $p_{(1,2),(1,2)}^{(2,2)}p_{(1,3),(2,2)}^{(2,2)}=p_{(1,3),(1,2)}^{(1,2)}p_{(1,2),(1,2)}^{(2,2)}.$ It follows that $p_{(1,3),(2,2)}^{(2,2)}=p_{(1,3),(1,2)}^{(1,2)}$.

Note that $p_{(1,3),(1,3)}^{(2,2)}=p_{(1,2),(2,2)}^{(2,2)}=0$. Since $(1,3)$ is mixed and $\Delta_I(x)$ is not semicomplete, from Lemma \ref{jb} \ref{jb-2} and Lemma \ref{2,2 1}, we obtain $\Gamma_{2,2}\Gamma_{1,2}\subseteq\{\Gamma_{2,1},\Gamma_{3,1}\}$ and $\Gamma_{2,2}\Gamma_{1,3}\subseteq\{\Gamma_{2,1},\Gamma_{2,2}\}$. By Lemma \ref{jb} \ref{jb-3}, one has
\begin{align}\label{3.9}
p_{(1,2),(1,2)}^{(2,2)}=k_{1,2}-p_{(1,2),(1,3)}^{(2,2)}=k_{1,2}-(k_{1,3}-p_{(1,3),(2,2)}^{(2,2)})=p_{(1,3),(2,2)}^{(2,2)}=p_{(1,3),(1,2)}^{(1,2)}.
\end{align}

	Let $y'\in P_{(1,2),(1,2)}(x,z)$. Note that $p_{(1,3),(1,2)}^{(1,2)}\neq0$. Pick a vertex $u\in P_{(1,3),(1,2)}(y',z)$. Since $p_{(1,3),(3,1)}^{(1,2)}=0$ from \eqref{666}, we get $u\notin P_{(1,3),(3,1)}(x,y')$. Since $\Gamma_{2,2}\Gamma_{1,2}\subseteq\{\Gamma_{2,1},\Gamma_{3,1}\}$ and $z\in P_{(1,2),(2,2)}(u,x)$, one has $(x,u)\in\Gamma_{1,2}$. Then $P_{(1,3),(1,2)}(y',z)\cup\{y'\}\subseteq P_{(1,2),(1,2)}(x,z)$, contrary to \eqref{3.9}.
	
	Since D$(4)$ does not exist, from Lemma \ref{1,3}, $p_{(1,2),(2,1)}^{(1,3)}=0$ and {\rm C}$(4)$ exists.
\end{proof}

\begin{lemma}\label{p(1,3),(3,1)^(1,1)=k_1,3}
	If $2\in T$ and $(1,3)$ is mixed, then $\Gamma_{1,1}\Gamma_{1,3}=\{\Gamma_{1,3}\}$.
\end{lemma}
\begin{proof}
	Let $(x,y)\in\Gamma_{1,1}$ and $(y,z)\in\Gamma_{1,3}$. Since $x,z\in N^+(y)$, we have $(x,z)\in\Gamma_{1,r}$ with $r>1$. Lemma \ref{D4 not hold} implies that $p_{(1,3),(1,3)}^{(1,2)}\neq0$ and $(1,2)$ is pure. By Lemma \ref{jb} \ref{jb-2}, one gets $p_{(1,2),(3,1)}^{(1,3)}\neq0$. Then there exists $w\in P_{(1,2),(3,1)}(y,z)$. Since $p_{(1,2),(2,1)}^{(1,1)}=k_{1,2}$ by Lemmas \ref{comm} and \ref{p(1,2),(2,1)^(1,1)=k_1,2}, we obtain $w\in P_{(1,2),(2,1)}(x,y)$. It follows that $r=\partial(z,x)\leq1+\partial(w,x)=3$. If $r=2$, then $z\in P_{(1,2),(3,1)}(x,y)$, contrary to the fact that $p_{(1,2),(2,1)}^{(1,1)}=k_{1,2}$. Then $r=3$, and so $\Gamma_{1,1}\Gamma_{1,3}=\{\Gamma_{1,3}\}$.
\end{proof}

\begin{lemma}\label{fuzhu3.1}
	If $(1,3)$ is mixed, then $\Gamma_{1,2}^2\subseteq\{\Gamma_{1,2},\Gamma_{2,1}\}$ and $p_{(1,2),(1,3)}^{(1,3)}=p_{(1,2),(2,1)}^{(1,2)}$.
\end{lemma}
\begin{proof}
By Lemma \ref{D4 not hold}, $p_{(1,2),(2,1)}^{(1,3)}=0$ and {\rm C}$(4)$ exists. Since $(1,2)$ is pure, there exists a circuit $(u,x,y)$ consisting of arcs of type $(1,2)$.  The fact $p_{(1,3),(1,3)}^{(1,2)}\neq0$ implies that there exists $z\in P_{(1,3),(1,3)}(y,u)$. It follows that $(x,z)\in\Gamma_{2,2}$. Suppose $p_{(1,2),(1,2)}^{(2,2)}\neq0$. It follows that there exists $y'\in P_{(1,2),(1,2)}(x,z)$. Since $x\in P_{(2,1),(1,2)}(y,y')$ and $p_{(1,2),(2,1)}^{(1,3)}=0$, from Lemma \ref{1,2^2} \ref{1,2 2,1}, we have $(y,y')\in\Gamma_{1,1}\cup\Gamma_{1,2}$. Since $z\in P_{(1,3),(2,1)}(y,y')$, from Lemma \ref{p(1,2),(2,1)^(1,1)=k_1,2}, one gets $(y,y')\in\Gamma_{1,2}$. The fact $y'\in P_{(1,2),(1,2)}(y,z)$ implies $p_{(1,2),(1,2)}^{(1,3)}\neq0$, contrary to Lemma \ref{jb2} \ref{jb2-1}. Hence, $p_{(1,2),(1,2)}^{(2,2)}=0$. It follows from Lemma \ref{1,2^2} \ref{1,2^2,1} and Lemma \ref{jb2} \ref{jb2-1} that $\Gamma_{1,2}^2\subseteq\{\Gamma_{1,2},\Gamma_{2,1}\}$.

If $w\in P_{(1,2),(2,1)}(x,y)$, then $(w,z)\in A(\Gamma)$ since $z,w\in N^+(y)$ and $(x,z)\in\Gamma_{2,2}$, which implies $w\in P_{(1,2),(1,3)}(y,z)$ from Lemma \ref{1,3 3,1} \ref{1,3 3,1-1} and Lemma \ref{jb2} \ref{jb2-1}. Then $p_{(1,2),(2,1)}^{(1,2)}\leq p_{(1,2),(1,3)}^{(1,3)}$. If $w'\in P_{(1,2),(1,3)}(y,z)$, then $(w',x)\notin A(\Gamma)$ since $\Gamma$ is locally semicomplete and $(x,z)\in\Gamma_{2,2}$, which implies $w'\in P_{(1,2),(2,1)}(x,y)$ by $\Gamma_{1,2}^2\subseteq\{\Gamma_{1,2},\Gamma_{2,1}\}$. Then $p_{(1,2),(1,3)}^{(1,3)}\leq p_{(1,2),(2,1)}^{(1,2)}$, and so $p_{(1,2),(1,3)}^{(1,3)}=p_{(1,2),(2,1)}^{(1,2)}$.
\end{proof}

		
\begin{proof}[Proof of Proposition~\ref{C4 holds}]
	 (i) By Lemma \ref{D4 not hold}, C$(4)$ exists. This proves \ref{C4-3}. 
	
	 (ii) By Lemma \ref{jb2} \ref{jb2-1} and Lemmas \ref{D4 not hold}, \ref{fuzhu3.1}, we have  $$p_{(1,2),(1,2)}^{(2,2)}=p_{(1,2),(1,2)}^{(1,3)}=p_{(1,2),(2,1)}^{(1,3)}=0.$$ It follows from Lemma \ref{jb} \ref{jb-2} that $$p_{(2,1),(2,2)}^{(1,2)}=p_{(2,1),(1,3)}^{(1,2)}=p_{(1,3),(1,2)}^{(1,2)}=0.$$ In view of Lemma \ref{fuzhu3.1}, one has $p_{(1,2),(1,3)}^{(1,3)}=p_{(1,2),(2,1)}^{(1,2)}$. By setting $\wz{d}=\wz{e}=(1,3)$ and $\wz{f}=\wz{g}^t=(1,2)$ in Lemma \ref{jb} \ref{jb-4}, from  Lemma \ref{1,3 3,1} \ref{1,2 1,3} and \ref{1,3 3,1-1}, we get $$p_{(1,3),(1,3)}^{(1,2)}p_{(1,2),(1,3)}^{(1,3)}=p_{(1,3),(1,3)}^{(1,2)}p_{(1,2),(2,1)}^{(1,2)}=p_{(2,1),(1,3)}^{(1,3)}p_{(1,3),(1,3)}^{(1,2)}+p_{(2,1),(1,3)}^{(3,1)}p_{(1,3),(3,1)}^{(1,2)}.$$  Since $p_{(1,2),(1,3)}^{(1,3)}=p_{(2,1),(1,3)}^{(1,3)}$ from Lemma \ref{jb} \ref{jb-2}, we have $p_{(2,1),(1,3)}^{(3,1)}p_{(1,3),(3,1)}^{(1,2)}=0$. Since $p_{(1,3),(1,3)}^{(1,2)}\neq0$, from Lemma \ref{jb} \ref{jb-2}, one has $p_{(2,1),(1,3)}^{(3,1)}\neq0$, and so $p_{(1,3),(3,1)}^{(1,2)}=0$. Since $p_{(1,3),(3,1)}^{(1,2)}=p_{(1,2),(1,2)}^{(1,3)}=p_{(1,2),(2,1)}^{(1,3)}=0,$ from Lemma \ref{jb} \ref{jb-2}, one gets $\Gamma_{3,1},\Gamma_{1,2},\Gamma_{2,1}\notin\Gamma_{1,2}\Gamma_{3,1}$. Lemma \ref{1,3 3,1} \ref{1,3 3,1-2} implies $\Gamma_{1,2}\Gamma_{3,1}=\{\Gamma_{1,3}\}$. This proves \ref{C4-2}.
	
	 (iii) By \ref{C4-2}, one gets $p_{(1,2),(3,1)}^{(3,1)}=0$. It follows from Lemma \ref{jb} \ref{jb-2} and Lemma \ref{fuzhu3.1} that $p_{(1,2),(1,3)}^{(1,3)}=p_{(1,2),(2,1)}^{(1,2)}=0$. Since $p_{(1,2),(2,1)}^{(1,3)}=0$ from Lemma \ref{D4 not hold}, by Lemma \ref{1,2^2} \ref{1,2 2,1}, one obtains $\Gamma_{1,2}\Gamma_{2,1}\subseteq\{\Gamma_{0,0},\Gamma_{1,1}\}$. Since $p_{(1,2),(2,1)}^{(1,1)}=k_{1,2}$ from Lemmas \ref{comm} and \ref{p(1,2),(2,1)^(1,1)=k_1,2}, we have $k_{1,2}=k_{1,1}+1$ by Lemma \ref{jb} \ref{jb-1}.
	
	Let $(x,y)\in\Gamma_{1,2}$ and $(y,z)\in\Gamma_{1,3}$. Since $p_{(1,3),(1,3)}^{(1,2)}\neq0$, from Lemma \ref{jb} \ref{jb-2}, we have $p_{(1,2),(3,1)}^{(1,3)}\neq0$, which implies that there exists $w\in P_{(1,2),(3,1)}(y,z)$. Since $p_{(1,2),(2,1)}^{(1,2)}=0$, from Lemma \ref{jb} \ref{jb-2}, one gets $\Gamma_{1,2}\notin\Gamma_{1,2}^2$. It follows from Lemma \ref{fuzhu3.1} that $\Gamma_{1,2}^2=\{\Gamma_{2,1}\}$, and so $(x,w)\in\Gamma_{2,1}$. Then $(z,x)\in\Gamma_{2,2}$, and so $\Gamma_{1,2}\Gamma_{1,3}=\{\Gamma_{2,2}\}$.
	
	Pick a vertex $y'\in\Gamma_{1,2}(x)$. If $y'=y$, then $y'\in P_{(1,2),(1,3)}(x,z)$. Suppose $y'\neq y$. Since $\Gamma_{1,2}\Gamma_{2,1}\subseteq\{\Gamma_{0,0},\Gamma_{1,1}\}$, we have $(y,y')\in\Gamma_{1,1}$. By Lemma \ref{p(1,3),(3,1)^(1,1)=k_1,3}, one gets $y'\in P_{(1,2),(1,3)}(x,z)$. It follows that $p_{(1,2),(1,3)}^{(2,2)}=k_{1,2}$. By setting $\wz{d}=(1,3)$ and $\wz{e}=(1,2)$ in Lemma \ref{jb} \ref{jb-1}, we obtain $k_{2,2}=k_{1,3}$.
	
	By Lemma \ref{jb} \ref{jb-2}, one gets $p_{(1,3),(1,3)}^{(2,2)}=p_{(2,2),(3,1)}^{(1,3)}$. Suppose $p_{(1,3),(1,3)}^{(2,2)}\neq0$. Then there exists a circuit $(x_0,x_1,x_2,x_3)$ consisting of arcs of type $(1,3)$ with $(x_1,x_3)\in\Gamma_{2,2}$. For any $y_2\in P_{(1,3),(1,2)}(x_1,x_3)\cup P_{(1,3),(1,3)}(x_1,x_3)$, we obtain $y_2\in P_{(2,2),(3,1)}(x_0,x_1)$. It follows that $P_{(1,3),(1,2)}(x_1,x_3)\cup P_{(1,3),(1,3)}(x_1,x_3)\subseteq P_{(2,2),(3,1)}(x_0,x_1)$, contrary to the fact that $p_{(1,3),(1,3)}^{(2,2)}=p_{(2,2),(3,1)}^{(1,3)}$ and $p_{(1,3),(1,2)}^{(2,2)}\neq0$. Thus, $p_{(1,3),(1,3)}^{(2,2)}=0$.
	
	By \ref{C4-2} and Lemma \ref{jb} \ref{jb-2}, we have $p_{(1,3),(3,1)}^{(1,2)}=p_{(1,2),(3,1)}^{(3,1)}=0$. Lemma \ref{1,3 3,1} \ref{1,3 3,1-3} implies $\Gamma_{1,3}\Gamma_{3,1}\subseteq\{\Gamma_{0,0},\Gamma_{1,1},\Gamma_{1,3},\Gamma_{3,1}\}$. Since $p_{(1,3),(3,1)}^{(1,1)}=k_{1,3}$ from Lemmas   \ref{comm} and \ref{p(1,3),(3,1)^(1,1)=k_1,3}, we get $p_{(1,3),(3,1)}^{(1,3)}=(k_{1,3}-k_{1,1}-1)/2$ by Lemma \ref{jb} \ref{jb-1}. In view of Lemma \ref{jb} \ref{jb-2}, we obtain $p_{(1,3),(1,3)}^{(1,3)}=(k_{1,3}-k_{1,1}-1)/2$.
	
By \ref{C4-2} and  Lemma \ref{comm}, one has $p_{(1,3),(1,3)}^{(1,2)}=k_{1,3}$. Lemma \ref{jb} \ref{jb-2} implies $p_{(1,2),(3,1)}^{(1,3)}=k_{1,2}$.  Since $p_{(1,3),(1,3)}^{(2,2)}=0$, from Lemma \ref{1,3 3,1} \ref{1,2 1,3}, we get $\Gamma_{1,3}^2\subseteq\{\Gamma_{1,2},\Gamma_{1,3}\}$.
Lemma \ref{jb} \ref{jb-1} implies $k_{1,3}^2=p_{(1,3),(1,3)}^{(1,3)}k_{1,3}+p_{(1,3),(1,3)}^{(1,2)}k_{1,2}.$ Since $p_{(1,3),(1,3)}^{(1,3)}=(k_{1,3}-k_{1,1}-1)/2$ and $k_{1,2}=k_{1,1}+1$, one obtains $k_{1,2}=k_{1,3}=k_{1,1}+1$. It follows that $\Gamma_{1,3}\notin\Gamma_{1,3}^2$, and so $\Gamma_{1,3}^2=\{\Gamma_{1,2}\}$. This proves \ref{C4-1}.
\end{proof}

\section{Arcs of type $(1,q-1)$ with $q\geq 5$}

In this section, we always assume that $\Gamma$ is a locally semicomplete commutative weakly distance-regular digraph. The main result is as follows, which characterizes mixed arcs of type $(1,q-1)$ with $q\geq 5$.
\begin{prop}\label{mixed}
Let $q\geq5$. Then $(1,q-1)$ is mixed if and only if {\rm C}$(q)$ exists.
\end{prop}

In order to prove Proposition \ref{mixed}, we need some auxiliary lemmas.

\begin{lemma}\label{pure}
Let $(1,q-1)$ be pure with $q>3$. Then $\Delta_q\simeq{\rm Cay}(\mathbb{Z}_q,\{1\})[\overline{K}_{k_{1,q-1}}]$, $\Gamma_{2,q-2}\Gamma_{q-1,1}=\{\Gamma_{1,q-1}\}$ and $\Gamma_{1,q-1}^l=\{\Gamma_{l,q-l}\}$ for $1\leq l<q$. Moreover,
		if $\Gamma_{1,r-1}\in\Gamma_{1,q-1}\Gamma_{q-1,1}$, then $\Gamma_{1,r-1}\Gamma_{1,q-1}=\{\Gamma_{1,q-1}\}$.

\end{lemma}
\begin{proof}
	Since $(1,q-1)$ is pure, there exists a circuit $(x_0,x_1,\ldots,x_{q-1})$ consisting of arcs of type $(1,q-1)$. It follows that $\partial(x_i,x_{i+j})=j$ for $1\leq j\leq q-1$, where the indices are read modulo $q$. If $k_{1,q-1}=1$, then the desired result follows.
	
	Now consider the case that $k_{1,q-1}>1$. Let $y\in\Gamma_{1,q-1}(x_0)$ such that $(x_1,y)\in A(\Gamma)$. Since $y,x_2\in N^+(x_1)$, $(y,x_2)$ or $(x_2,y)$ is an arc. Since $(x_0,x_2)\in\Gamma_{2,q-2}$ with $q>3$, one gets $(x_2,y)\notin A(\Gamma)$, and so $(y,x_{2})\in A(\Gamma)$. The fact $(1,q-1)$ is pure implies $(y,x_2)\in\Gamma_{1,q-1}$. It follows that $\cup_i P_{(1,q-1),(i,1)}(x_0,x_1)\subseteq P_{(1,q-1),(1,q-1)}(x_0,x_2)$.
	
	Let $\Gamma_{1,i}\in\Gamma_{q-1,1}\Gamma_{1,q-1}$ with $i\geq1$. Suppose that there exists a vertex $z$ such that $(z,x_1)\in\Gamma_{1,i}$ and $(x_0,z)\notin\Gamma_{1,q-1}$. Since $x_0,z\in N^-(x_1)$, $(x_0,z)$ or $(z,x_0)$ is an arc. If $(x_0,z)$ is an arc, then $(z,x_2)\notin\Gamma_{1,q-1}$ since $(1,q-1)$ is pure and $(x_0,z)\notin\Gamma_{1,q-1}$; if $(z,x_0)$ is an arc, then $(z,x_2)\notin\Gamma_{1,q-1}$ since $\Gamma$ is locally semipcomlete and $(x_0,x_2)\in\Gamma_{2,q-2}$ with $q>3$. Hence, $$\Gamma_{i,1}(x_1)\setminus P_{(1,q-1),(1,i)}(x_0,x_1)\subseteq\Gamma_{i,1}(x_1)\setminus P_{(i,1),(1,q-1)}(x_1,x_2).$$ By Lemma \ref{jb} \ref{jb-2}, one has $p_{(1,q-1),(1,i)}^{(1,q-1)}=p_{(i,1),(1,q-1)}^{(1,q-1)}$, and so $$\Gamma_{i,1}(x_1)\setminus P_{(1,q-1),(1,i)}(x_0,x_1)=\Gamma_{i,1}(x_1)\setminus P_{(i,1),(1,q-1)}(x_1,x_2).$$ Then $P_{(1,q-1),(1,i)}(x_0,x_1)=P_{(i,1),(1,q-1)}(x_1,x_2)$. Since $\Gamma_{1,i}\in\Gamma_{q-1,1}\Gamma_{1,q-1}$ was arbitrary, we get $\cup_i P_{(1,q-1),(1,i)}(x_0,x_1)\subseteq P_{(1,q-1),(1,q-1)}(x_0,x_2)$.
	
	Since $\Gamma$ is locally semicomplete, one obtains $\wz{i}=(1,i)$ or $(i,1)$ with $i>0$ for all $\Gamma_{\wz{i}}\in\Gamma_{1,q-1}\Gamma_{q-1,1}\setminus\{\Gamma_{0,0}\}$. The fact $\cup_i P_{(1,q-1),(i,1)}(x_0,x_1)\subseteq P_{(1,q-1),(1,q-1)}(x_0,x_2)$ implies $\Gamma_{1,q-1}(x_0)\subseteq P_{(1,q-1),(1,q-1)}(x_0,x_2)$, and so $p_{(1,q-1),(1,q-1)}^{(2,q-2)}=k_{1,q-1}$. Then the first statement is valid from \cite[Lemma 2.4]{YYF22} and Lemma \ref{comm}.
	
Now suppose $\Gamma_{1,r-1}\in\Gamma_{1,q-1}\Gamma_{q-1,1}$. Let $(x,y)\in\Gamma_{1,r-1}$ and $(y,z)\in \Gamma_{1,q-1}$. By the commutativity of  $\Gamma$, there exists $w\in P_{(q-1,1),(1,q-1)}(x,y)$.  From the first statement, one gets $\Gamma_{1,q-1}^2=\{\Gamma_{2,q-2}\}$ and $p_{(1,q-1),(1,q-1)}^{(2,q-2)}=k_{1,q-1}$, which imply $(w,z)\in\Gamma_{2,q-2}$ and $(x,z)\in \Gamma_{1,q-1}$. Thus, the second statement is also valid.
\end{proof}

\begin{lemma}\label{xingjia}
	Let $(1,q-1)$ be pure with $q>3$. If $\Gamma_{1,r}\in\Gamma_{1,q-1}\Gamma_{q-1,1}$, then $r\in\{1,2\}$.
\end{lemma}
\begin{proof}
	For all $\Gamma_{1,s}\in\Gamma_{1,q-1}\Gamma_{q-1,1}$, by Lemma \ref{pure}, one gets $\Gamma_{1,s}\Gamma_{1,q-1}=\{\Gamma_{1,q-1}\}$. It follows from Lemmas \ref{diameter 2} and \ref{bkb} that $r\in\{1,2\}$.
\end{proof}

\begin{lemma}\label{mixed-2}
If {\rm C}$(q)$ exists with $q\geq5$, then $\Gamma_{1,q-1}\Gamma_{1,q-2}=\{\Gamma_{2,q-2}\}$.
\end{lemma}
\begin{proof}
Pick a path $(x,x_1,x_2)$ such that $(x,x_1)\in\Gamma_{1,q-1}$ and $(x_1,x_{2})\in\Gamma_{1,q-2}$. Since $p_{(1,q-1),(1,q-1)}^{(1,q-2)}\neq0$, from Lemma \ref{jb} \ref{jb-2}, we have $p_{(q-1,1),(1,q-2)}^{(1,q-1)}\neq0$. Then there exists $x_0\in P_{(q-1,1),(1,q-2)}(x,x_1)$. Since $(1,q-2)$ is pure, from Lemma \ref{pure}, we have $\Gamma_{1,q-2}^2=\{\Gamma_{2,q-3}\}$, which implies $(x_0,x_{2})\in\Gamma_{2,q-3}$. The fact $q-2\leq\partial(x_2,x)\leq\partial(x_2,x_0)+1=q-2$ implies $\partial(x_2,x)=q-2$. By Lemma \ref{not pure}, one gets $P_{(1,q-1),(1,q-2)}(x_0,x_2)=\emptyset$, and so $(x,x_2)\in\Gamma_{2,q-2}$. This completes the proof of this lemma.
\end{proof}

\begin{lemma}\label{mixed-1}
If {\rm C}$(q)$ exists with $q\geq5$, then $\Gamma_{1,q-2}\Gamma_{q-1,1}=\{\Gamma_{1,q-1}\}$.
\end{lemma}
\begin{proof}
Pick a path $(x_0,x_1,x_2)$ consisting of arcs of type $(1,q-2)$. Since $p_{(1,q-1),(1,q-1)}^{(1,q-2)}\neq0$, there exists $y\in P_{(1,q-1),(1,q-1)}(x_0,x_1)$. By Lemma \ref{pure}, we have $(x_0,x_2)\in\Gamma_{2,q-3}$.

Assume the contrary, namely, there exists $z\in\Gamma_{1,q-1}(x_0)$ such that $(z,x_1)\notin\Gamma_{1,q-1}$. Since $z,x_1\in N^+(x_0)$, we have $(z,x_1)$ or $(x_1,z)\in A(\Gamma)$. Suppose $(x_1,z)\in A(\Gamma)$. The fact $z,x_2\in N^+(x_1)$ implies that $(z,x_2)$ or $(x_2,z)\in A(\Gamma)$. Since $(x_0,z)\in A(\Gamma)$ and $(x_0,x_2)\in\Gamma_{2,q-3}$ with $q>4$, one gets $(z,x_2)\in\Gamma_{1,r}$ for some $r>0$. Since $z\in P_{(1,q-1),(1,r)}(x_0,x_2)$, we obtain $p_{(1,q-1),(1,r)}^{(2,q-3)}\neq0$, contrary to Lemma \ref{not pure}. Then $(z,x_1)\in\Gamma_{1,r}$ with $r\neq q-1$.

If $r=q-2$, then $x_1\in P_{(1,q-2),(q-2,1)}(x_0,z)$, 
which implies $\Gamma_{1,q-1}\in \Gamma_{1,q-2}\Gamma_{q-2,1}$, contrary to Lemma \ref{xingjia}. 
Since $r=\partial(x_1,z)\leq\partial(x_1,x_0)+1$, we obtain $r<q-2$. Since $\partial(x_2,z)\leq\partial(x_2,x_0)+1=q-2$ and $(1,q-2)$ is pure, we have $\partial(x_2,z)=q-2$, and so $\partial(z,x_2)=2$. By Lemma \ref{mixed-2}, one has  $(y,x_2)\in\Gamma_{2,q-2}$. The fact $x_1\in P_{(1,r),(1,q-2)}(z,x_2)$ implies that there exists $x_1'\in P_{(1,r),(1,q-2)}(y,x_2)$. Since $x_1,x_1'\in N^+(y)$, one gets $(x_1,x_1')$ or $(x_1',x_1)\in A(\Gamma)$. If $(x_1,x_1')\in A(\Gamma)$, then $q-1=\partial(x_1,y)\leq1+\partial(x_1',y)=r+1<q-1,$ a contradiction. Then $(x_1',x_1)\in\Gamma_{1,s}$ with $s>1$.

Since $x_2\in P_{(1,q-2),(q-2,1)}(x_1',x_1)$,  from Lemma \ref{xingjia}, we get $s=2$. Since $$q-1=\partial(x_1,y)\leq 2+\partial(x_1',y)=r+2\leq q-1,$$ we have $r=q-3$.

Since $(x_1',x_1)\in\Gamma_{1,2}$, there exists a path $(x_1,x_2',x_1')$. Since $(y,x_1)\in\Gamma_{1,q-1}$, $(y,x_1,x_2')$ is not a circuit, which implies $(x_2',y)\notin A(\Gamma)$. Since $y,x_2'\in N^-(x_1')$, we have $(y,x_2')\in A(\Gamma)$. Since $\partial(x_2',y)\leq 1+\partial(x_1',y)=q-2$, one gets $$q-1=\partial(x_1,y)\leq1+\partial(x_2',y)\leq q-1,$$ and so $(y,x_2')\in\Gamma_{1,q-2}$. The fact $x_1'\in P_{(1,q-3),\wz{\partial}(x_1',x_2')}(y,x_2')$ implies that there exists $w\in\Gamma_{1,q-3}(x_0)$ such that $(x_1,w)\in A(\Gamma)$. Since $w,x_2\in N^+(x_1)$ and $(x_0,x_2)\in\Gamma_{2,q-3}$ with $q\geq5$, we get $(w,x_2)\in\Gamma_{1,t}$ for some $t>0$. Since $w\in P_{(1,q-3),(1,t)}(x_0,x_2)$, we obtain $p_{(1,q-3),(1,t)}^{(2,q-3)}\neq0$, contrary to Lemma \ref{not pure}.
\end{proof}


\begin{lemma}\label{mixed-3}
Let $r\neq q-1$ with $q\geq5$. Suppose that {\rm C}$(q)$ exists. The following hold:
\begin{enumerate}
	\item\label{mixed-3-1} $\Gamma_{1,r}\in\Gamma_{1,q-2}\Gamma_{q-2,1}$ if and only if $\Gamma_{1,r}\in\Gamma_{1,q-1}\Gamma_{q-1,1}$;
	
	\item\label{mixed-3-2} If $\Gamma_{1,r}\in\Gamma_{1,q-1}\Gamma_{q-1,1}$, then $\Gamma_{1,r}\Gamma_{1,q-1}=\{\Gamma_{1,q-1}\}$.
\end{enumerate}
\end{lemma}
\begin{proof}
	 We claim that $\Gamma_{1,r}\in\Gamma_{1,q-1}\Gamma_{q-1,1}$ and $p_{(1,q-1),(q-1,1)}^{(1,r)}=k_{1,q-1}$ for each $\Gamma_{1,r}\in\Gamma_{1,q-2}\Gamma_{q-2,1}$. Let $(x,z)\in\Gamma_{1,r}$ and $y\in P_{(1,q-2),(q-2,1)}(x,z)$. Pick a vertex $w\in\Gamma_{1,q-1}(x)$. By Lemmas \ref{comm} and \ref{mixed-1}, one has $p_{(1,q-1),(1,q-1)}^{(1,q-2)}=k_{1,q-1}$, which implies that $w\in P_{(1,q-1),(1,q-1)}(x,y)$ and $w\in P_{(1,q-1),(1,q-1)}(z,y)$. Then $\Gamma_{1,r}\in\Gamma_{1,q-1}\Gamma_{q-1,1}$ and $p_{(1,q-1),(q-1,1)}^{(1,r)}=k_{1,q-1}$. Thus, our claim is valid.

(i) Suppose $\Gamma_{1,r}\in\Gamma_{1,q-1}\Gamma_{q-1,1}$. Let $(x,z)\in\Gamma_{1,r}$ and $y\in P_{(1,q-1),(q-1,1)}(x,z)$. Pick a vertex $w\in\Gamma_{1,q-2}(x)$. Since $p_{(1,q-1),(1,q-1)}^{(1,q-2)}=k_{1,q-1}$, we have $y\in P_{(1,q-1),(1,q-1)}(x,w)$. Since $z,w\in N^+(x)$, we have $(z,w)$ or $(w,z)\in A(\Gamma)$. The fact that $(z,y,w)$ is a path consisting of arcs of type $(1,q-1)$ implies $(z,w)\in A(\Gamma)$. Since $r\neq q-1$ and $p_{(1,q-1),(1,q-1)}^{(1,q-2)}=k_{1,q-1}$, we obtain $(z,w)\notin\Gamma_{1,q-1}$. Since $q-2=\partial(y,z)-1\leq\partial(w,z)\leq\partial(w,x)+1=q-1,$ one has $(z,w)\in\Gamma_{1,q-2}$. The fact $w\in P_{(1,q-2),(q-2,1)}(x,z)$ implies $\Gamma_{1,r}\in\Gamma_{1,q-2}\Gamma_{q-2,1}$. Combining with the claim, \ref{mixed-3-1} is valid.

(ii) By \ref{mixed-3-1}, we have $\Gamma_{1,r}\in\Gamma_{1,q-2}\Gamma_{q-2,1}$, which implies $p_{(1,q-1),(q-1,1)}^{(1,r)}=k_{1,q-1}$ from the claim. It follows from Lemma \ref{comm} that $\Gamma_{1,r}\Gamma_{1,q-1}=\{\Gamma_{1,q-1}\}$. Thus, \ref{mixed-3-2} is valid.
\end{proof}

\begin{lemma}\label{mixed-4}
If {\rm C}$(q)$ exists with $q\geq5$, then $\Gamma_{1,q-1}\notin\Gamma_{1,q-1}^2$.
\end{lemma}
\begin{proof}
Pick a path $(x,y,z)$ such that $(x,y)\in\Gamma_{1,q-1}$ and $(y,z)\in\Gamma_{1,q-2}$. By Lemma \ref{mixed-2}, we have $(x,z)\in\Gamma_{2,q-2}$. Assume the contrary, namely, $\Gamma_{1,q-1}\in\Gamma_{1,q-1}^2$. In view of Lemma \ref{jb} \ref{jb-2}, one gets $p_{(1,q-1),(q-1,1)}^{(1,q-1)}\neq0$, which implies that there exists $y'\in P_{(1,q-1),(q-1,1)}(x,y)$. It follows from Lemmas \ref{comm} and \ref{mixed-1} that $y'\in P_{(1,q-1),(1,q-1)}(y,z)$. Since $y'\in P_{(1,q-1),(1,q-1)}(x,z)$, we obtain $\Gamma_{2,q-2}\in\Gamma_{1,q-1}^2$.

Pick a vertex $w\in P_{(1,q-1),(q-1,1)}(x,y')$. Since $w,z\in N^+(y')$ and $(x,z)\in\Gamma_{2,q-2}$ with $q\geq5$, we have $(w,z)\in A(\Gamma)$. Since $y'\in P_{(q-1,1),(1,q-1)}(w,z)$, from Lemma \ref{mixed-3} \ref{mixed-3-2}, one gets $P_{(q-1,1),(1,q-1)}(w,z)=\Gamma_{q-1,1}(w)$ or $(w,z)\in\Gamma_{1,q-1}$. The fact $x\in P_{(q-1,1),(2,q-2)}(w,z)$ implies $(w,z)\in\Gamma_{1,q-1}$. It follows that $w\in P_{(1,q-1),(1,q-1)}(x,z)$. Then $P_{(1,q-1),(q-1,1)}(x,y')\subseteq P_{(1,q-1),(1,q-1)}(x,z)$.
Pick a vertex $w'\in\Gamma_{1,q-1}(x)$ such that $(w',y')\notin\Gamma_{1,q-1}\cup\Gamma_{q-1,1}$. Since $y',w'\in N^+(x)$, one gets $(w',y')$ or $(y',w')\in A(\Gamma)$. Since $x\in P_{(q-1,1),(1,q-1)}(y',w')$, from Lemma \ref{comm} and Lemma \ref{mixed-3} \ref{mixed-3-2}, we obtain $z\in P_{(1,q-1),(q-1,1)}(y',w')$, and so $w'\in P_{(1,q-1),(1,q-1)}(x,z)$. Since $P_{(1,q-1),(q-1,1)}(x,y')\subseteq P_{(1,q-1),(1,q-1)}(x,z)$, we have $$\Gamma_{1,q-1}(x)\setminus P_{(1,q-1),(1,q-1)}(x,y')\subseteq P_{(1,q-1),(1,q-1)}(x,z),$$ and so $$p_{(1,q-1),(1,q-1)}^{(2,q-2)}\geq k_{1,q-1}-p_{(1,q-1),(1,q-1)}^{(1,q-1)}.$$

Let $(x_0,x_1,x_2,x_3)$ be a path such that $(x_0,x_2)\in\Gamma_{2,q-2}$, $x_1\in P_{(1,q-1),(1,q-1)}(x_0,x_2)$ and $\partial(x_3,x_0)=q-3$. Then $(x_1,x_3)\in\Gamma_{2,q-2}$. Let $x_2'\in P_{(1,q-1),(1,q-1)}(x_1,x_3)$. It follows that $(x_0,x_2')\in\Gamma_{2,q-2}$. Since $x_2'$ was arbitrary, we have $P_{(1,q-1),(1,q-1)}(x_1,x_3)\subseteq P_{(2,q-2),(q-1,1)}(x_0,x_1)$, which implies that  $$p_{(2,q-2),(q-1,1)}^{(1,q-1)}\geq p_{(1,q-1),(1,q-1)}^{(2,q-2)}\geq k_{1,q-1}-p_{(1,q-1),(1,q-1)}^{(1,q-1)}.$$ By Lemma \ref{jb} \ref{jb-2}, one has $$p_{(2,q-2),(q-1,1)}^{(1,q-1)}+p_{(1,q-1),(q-1,1)}^{(1,q-1)}+p_{(1,q-2),(q-1,1)}^{(1,q-1)}>k_{1,q-1},$$ contrary to Lemma \ref{jb} \ref{jb-3}.
\end{proof}

\begin{lemma}\label{mixed-5}
If {\rm C}$(q)$ exists with $q\geq5$, then $\Gamma_{1,q-1}^2=\{\Gamma_{1,q-2}\}$.
\end{lemma}
\begin{proof}
If $k_{1,q-1}=1$, then $\Gamma_{1,q-1}^2=\{\Gamma_{1,q-2}\}$. Now suppose $k_{1,q-1}>1$. Let $(x,z)\in\Gamma_{1,q-2}$ and $y\in P_{(1,q-1),(1,q-1)}(x,z)$. Pick a vertex $w\in\Gamma_{q-1,1}(y)$ with $x\ne w$. It suffices to show that $(w,z)\in\Gamma_{1,q-2}$. Since $x,w\in N^-(y)$, from Lemma \ref{mixed-4}, we have $(x,w)\in\Gamma_{1,r}\cup\Gamma_{r,1}$ with $r\neq q-1$, and so $\Gamma_{1,r}\in\Gamma_{1,q-1}\Gamma_{q-1,1}$. Lemma \ref{mixed-3} \ref{mixed-3-1} implies $\Gamma_{1,r}\in\Gamma_{1,q-2}\Gamma_{q-2,1}$. By Lemmas \ref{comm} and \ref{pure}, we get $p_{(1,q-2),(q-2,1)}^{(1,r)}=k_{q-2,1}$, and so $z\in P_{(1,q-2),(q-2,1)}(x,w)$. The desired result follows.
\end{proof}


Now we are ready to give a proof of Proposition \ref{mixed}.

\begin{proof}[Proof of Proposition~\ref{mixed}]
If {\rm C}$(q)$ exists, it is obvious that $(1,q-1)$ is mixed. We prove the converse.
By way of contradiction, we may assume that $q$ is the minimum positive integer such that $q\geq5$, $(1,q-1)$ is mixed and C$(q)$ does not exist. Pick a circuit $(x_0,x_1,\ldots,x_{q-1})$ such that $(x_{q-1},x_0)\in\Gamma_{1,q-1}$.

\textbf{Case 1.} $\Gamma_{1,q-1}\in\Gamma_{l-1,1}^{q-1}$ for some $l\in\{2,3,\ldots,q-1\}$.

Without loss of generality, we may assume $(x_i,x_{i+1})\in\Gamma_{1,l-1}$ for $0\leq i\leq q-2$. Suppose $l\geq4$. By the minimality of $q$ and Proposition \ref{C4 holds} \ref{C4-3}, $(1,l-1)$ is pure or C$(l)$ exists. If C$(l)$ exists, from Proposition \ref{C4 holds} \ref{C4-1} and Lemma \ref{mixed-5}, then $\partial(x_0,x_2)=1$, a contradiction. Hence, $(1,l-1)$ is pure. By Lemma \ref{pure}, one gets $\Gamma_{1,l-1}^{l-1}=\{\Gamma_{l-1,1}\}$ and $p_{(1,l-1),(1,l-1)}^{(2,l-2)}=k_{1,l-1}$ from Lemma \ref{comm}. Then $(x_0,x_{l-1})\in\Gamma_{l-1,1}$. If $l<q-1$, from Lemma \ref{pure}, then $(x_{l-1},x_{l+1})\in\Gamma_{2,l-2}$, and so $x_0\in P_{(1,l-1),(1,l-1)}(x_{l-1},x_{l+1})$, a contradiction. Then $l=q-1$. Since $x_{q-2}\in P_{(q-2,1),(1,q-2)}(x_{q-1},x_{0})$, from Lemma \ref{xingjia}, one gets $q-1\in\{1,2\}$, a contradiction. Thus, $l\leq3$.

If $l=2$, or $l=3$ and $(1,2)$ is mixed, from Lemmas \ref{1,1} and \ref{(1,2) mixed}, then $\Gamma_{1,q-1}\in F_l\subseteq\{\Gamma_{0,0},\Gamma_{1,2},\Gamma_{1,1},\Gamma_{2,1}\}$, a contradiction. Then $l=3$ and $(1,2)$ is pure.

Since $\partial(x_0,x_2)=2$, from Lemma \ref{1,2^2} \ref{1,2^2,1}, we have $(x_0,x_2)\in\Gamma_{2,i}$ for some $i\in\{1,2\}$. Since $\Gamma_{1,q-1}\in F_3$, from Lemmas \ref{semicoplete} and \ref{diameter 2}, $\Delta_I(x_0)$ is not semicomplete, where $I=\{r\mid\Gamma_{1,r-1}\in F_3\}$. 
In view of Lemma \ref{1,2^2} \ref{1,2 2,1} and Lemma \ref{2,2 1} \ref{2,2 1-1}, one gets $$\Gamma_{2,i}\Gamma_{1,2}\subseteq\{\Gamma_{0,0},\Gamma_{1,1},\Gamma_{1,2},\Gamma_{2,1},\Gamma_{1,3},\Gamma_{3,1},\Gamma_{2,2}\}.$$ Since $\partial(x_0,x_3)=3$, we obtain $(x_0,x_3)\in\Gamma_{3,1}$. Lemma \ref{1,3 3,1} \ref{1,3 3,1-2} implies $\partial(x_0,x_4)<4$, a contradiction.

\textbf{Case 2.} $\Gamma_{1,q-1}\notin\Gamma_{l-1,1}^{q-1}$ for all $l\in\{2,3,\ldots,q-1\}$.

Without loss of generality, we may assume $(x_0,x_1)\in\Gamma_{1,p-1}$ with $p<q$. The fact $x_0\in P_{(1,q-1),(1,p-1)}(x_{q-1},x_1)$ implies that there exists $x_0'\in P_{(1,p-1),(1,q-1)}(x_{q-1},x_1)$.

Assume that $(x_{q-1},x_1)\in\Gamma_{1,q-2}$. By the minimality of $q$ and Proposition \ref{C4 holds} \ref{C4-3}, $(1,q-2)$ is pure or C$(q-1)$ exists. Since $p_{(1,q-2),(1,q-2)}^{(2,q-3)}\neq0$ or $p_{(1,q-2),(1,q-2)}^{(1,q-3)}\neq0$, from Lemma \ref{jb} \ref{jb-2}, there exists $x_2'\in\Gamma_{1,q-2}(x_1)$ such that $\partial(x_2',x_{q-1})=q-3$. Suppose that $(1,q-2)$ is pure. Lemma \ref{pure} implies $(x_{q-1},x_2')\in\Gamma_{2,q-3}$. If $p=q-1$, then $x_{q-1}\in P_{(q-2,1),(1,q-2)}(x_0',x_1)$, which implies $q\in\{2,3\}$ from Lemma \ref{xingjia}. Then $p\neq q-1$. Since $(1,q-2)$ is pure, one gets $\partial(x_2',x_0')=\partial(x_2',x_0)=q-2$, and so $(x_0,x_2'),(x_0',x_2')\in\Gamma_{2,q-2}$. The fact $x_1\in P_{(1,q-1),(1,q-2)}(x_0',x_2')$ implies that there exists $x_1'\in P_{(1,q-1),(1,q-2)}(x_0,x_2')$. Since $p_{(1,q-2),(1,q-2)}^{(2,q-3)}=k_{1,q-2}$ from Lemmas \ref{comm} and \ref{pure}, we obtain $x_1'\in P_{(1,q-2),(1,q-2)}(x_{q-1},x_2')$, which implies that C$(q)$ exists, a contradiction. Suppose that C$(q-1)$ exists. By Lemma \ref{mixed-5}, we get $(x_{q-1},x_2')\in\Gamma_{1,q-3}$. Since $x_0,x_2'\in N^+(x_{q-1})$, one has $(x_0,x_2')\in A(\Gamma)$ or $(x_2',x_0)\in A(\Gamma)$, contrary to the fact that $\partial(x_0,x_2')=2$ or $q-2=\partial(x_2',x_1)\leq1+\partial(x_2',x_0)$. Thus, $(x_{q-1},x_1)\in\Gamma_{2,q-2}$.

Suppose $\Gamma_{1,q-2}\notin\Gamma_{1,q-1}^2$. If $(x_{q-2},x_{q-1})\in\Gamma_{1,q-1}$, then $(x_{q-2},x_{0})\in\Gamma_{2,q-2}$; if $(x_{q-2},x_{q-1})\notin\Gamma_{1,q-1}$, by similar argument, then $(x_{q-2},x_0)\in\Gamma_{2,q-2}$. Since $x_0\in P_{(1,q-1),(1,p-1)}(x_{q-1},x_1)$, there exists $x_{q-1}'\in P_{(1,q-1),(1,p-1)}(x_{q-2},x_0)$. Note that the circuit $(x_{q-1}',x_0,x_1,\ldots,x_{q-2})$ contains an arc of type $(1,q-1)$ and at least two arcs of type $(1,p-1)$. Repeat this process, there exists a circuit of length $q$ consisting of an arc of type $(1,q-1)$ and $q-1$ arcs of type $(1,p-1)$, contrary to the fact that $\Gamma_{1,q-1}\notin\Gamma_{1,p-1}^{q-1}$.

Since $\Gamma_{1,q-2}\in\Gamma_{1,q-1}^2$, there exists $x_{q-2}'\in\Gamma_{q-1,1}(x_{q-1})$ such that $(x_{q-2}',x_{0})\in\Gamma_{1,q-2}$. Since C$(q)$ does not exist, $(1,q-2)$ is mixed. By the minimality of $q$ and Proposition \ref{C4 holds} \ref{C4-3}, C$(q-1)$ exists. In view of Lemma \ref{jb} \ref{jb-2}, we have $p_{(1,q-3),(q-2,1)}^{(1,q-2)}\neq0$, which implies that there exists $x_1'\in P_{(1,q-3),(q-2,1)}(x_{q-2}',x_0)$. Since $x_{q-1},x_{1}'\in N^+(x_{q-2}')$, one gets $(x_{q-1},x_1')\in A(\Gamma)$ or $(x_1',x_{q-1})\in A(\Gamma)$. If $(x_{q-1},x_1')\in A(\Gamma)$, then $q-1=\partial(x_{q-1},x_{q-2}')\leq1+\partial(x_1',x_{q-2}')=q-2$, a contradiction. If $(x_1',x_{q-1})\in A(\Gamma)$, then $(x_{q-1},x_0,x_1')$ is a circuit containing an arc of type $(1,q-1)$, a contradiction.
\end{proof}

The following result is  immediate from Propositions  \ref{C4 holds}, \ref{mixed} and Lemmas \ref{mixed-1}, \ref{mixed-5}.

\begin{cor}\label{main cor}
Let $q\geq4$. Then $(1,q-1)$ is mixed if and only if {\rm C}$(q)$ exists. Moreover, 	if {\rm C}$(q)$ exists, then the following hold:
\begin{enumerate}
	\item\label{main cor-1} $\Gamma_{1,q-1}^2=\{\Gamma_{1,q-2}\}$;
	
	\item\label{main cor-2} $\Gamma_{1,q-2}\Gamma_{q-1,1}=\{\Gamma_{1,q-1}\}$.
\end{enumerate}
\end{cor}

\section{Subdigraphs}
Let $\Gamma$ be a locally semicomplete commutative weakly distance-regular digraph. For a nonempty subset $I$ of $T$ and $x\in V(\Gamma)$, recall the notation of the set $F_I$ and the digraph $\Delta_I(x)$ in Section 2. In this section, we focus on the existence of some special subdigraphs $\Delta_I(x)$ of $\Gamma$.

\begin{lemma}\label{delta3}
Let $(1,2)$ be pure. Suppose that $4\notin T$ or $(1,3)$ is pure. Then  $F_3\subseteq\{\Gamma_{0,0},\Gamma_{1,1},\Gamma_{1,2},\Gamma_{2,1},\Gamma_{2,2}\}$. Moreover, if $\Gamma_{2,2}\in F_3$, then $\Gamma_{2,2}\in\Gamma_{1,2}^2$.
\end{lemma}
\begin{proof}
If $F_3\subseteq\{\Gamma_{0,0},\Gamma_{1,1},\Gamma_{1,2},\Gamma_{2,1}\}$, then the desired result follows. Now we consider the case $F_3\nsubseteqq\{\Gamma_{0,0},\Gamma_{1,1},\Gamma_{1,2},\Gamma_{2,1}\}$. Let $I=\{r\mid\Gamma_{1,r-1}\in F_3\}$.

We claim that $\Gamma_{1,3}\notin F_3$. Suppose not. By Lemmas \ref{semicoplete} and \ref{diameter 2}, $\Delta_I(x)$ is not semicomplete for $x\in V(\Gamma)$. 
Lemma \ref{semi} \ref{semi-1} implies $\Gamma_{2,2}\in\Gamma_{1,2}^2\cup\Gamma_{1,2}\Gamma_{1,3}$. By Lemma \ref{not pure}, $(1,3)$ is mixed, a contradiction. Thus, our claim is valid.

By Lemma \ref{1,2^2} \ref{1,2^2,1} and the claim, we have $\Gamma_{1,2}^2\subseteq\{\Gamma_{2,1},\Gamma_{1,2},\Gamma_{2,2}\}$. In view of Lemma \ref{1,2^2} \ref{1,2 2,1}, one gets $\Gamma_{1,2}\Gamma_{2,1}\subseteq\{\Gamma_{0,0},\Gamma_{1,1},\Gamma_{1,2},\Gamma_{2,1}\}$. Suppose $\Gamma_{2,2}\notin\Gamma_{1,2}^2$. Then $\Gamma_{1,2}^2\subseteq\{\Gamma_{2,1},\Gamma_{1,2}\}$, and so $\Gamma_{1,2}^3\subseteq\{\Gamma_{0,0},\Gamma_{1,1},\Gamma_{1,2},\Gamma_{2,1}\}$. By induction and Lemma \ref{p(1,2),(2,1)^(1,1)=k_1,2}, we obtain $\Gamma_{1,2}^i\subseteq\{\Gamma_{0,0},\Gamma_{1,1},\Gamma_{1,2},\Gamma_{2,1}\}$ for $i\geq3$. It follows that $F_3\subseteq\{\Gamma_{0,0},\Gamma_{1,1},\Gamma_{1,2},\Gamma_{2,1}\}$, a contradiction. Then $\Gamma_{2,2}\in\Gamma_{1,2}^2$.

Since $\{\Gamma_{2,1},\Gamma_{2,2}\}\subseteq\Gamma_{1,2}^2\subseteq\{\Gamma_{1,2},\Gamma_{2,1},\Gamma_{2,2}\}$, $\Delta_I(x)$ is not semicomplete for $x\in V(\Gamma)$. Lemma \ref{2,2 1} \ref{2,2 1-1} and the claim imply $\Gamma_{2,2}\Gamma_{1,2}\subseteq\{\Gamma_{2,1},\Gamma_{2,2}\}$. Since $\Gamma_{1,2}\Gamma_{2,1}\subseteq\{\Gamma_{0,0},\Gamma_{1,1},\Gamma_{1,2},\Gamma_{2,1}\}$, one has $\Gamma_{1,2}^3\subseteq\{\Gamma_{0,0},\Gamma_{1,1},\Gamma_{1,2},\Gamma_{2,1},\Gamma_{2,2}\}$. By induction and Lemma \ref{p(1,2),(2,1)^(1,1)=k_1,2}, we get $\Gamma_{1,2}^i\subseteq\{\Gamma_{0,0},\Gamma_{1,1},\Gamma_{1,2},\Gamma_{2,1},\Gamma_{2,2}\}$ for $i\geq3$. The desired result follows.
\end{proof}

\begin{lemma}\label{k22=1}
 Let $(1,2)$ be pure and $2\notin T$. Suppose that $4\notin T$ or $(1,3)$ is pure.  If $\Gamma_{2,2}\in\Gamma_{1,2}^2$, then the following hold:
\begin{enumerate}
	\item\label{k22-1} $p_{(1,2),(1,2)}^{(1,2)}=p_{(1,2),(1,2)}^{(2,1)}=(k_{1,2}-1)/2$;
		
	\item\label{k22-2} $p_{(1,2),(1,2)}^{(2,2)}=k_{1,2}$;
	
		\item\label{k22-3} $k_{2,2}=1$.
\end{enumerate}
\end{lemma}
\begin{proof}
We claim that $p_{(2,2),(2,2)}^{(2,2)}=0$. Assume the contrary, namely, $p_{(2,2),(2,2)}^{(2,2)}\neq0$. Let $(x_0,x_4)\in\Gamma_{2,2}$ and $x_2\in P_{(2,2),(2,2)}(x_0,x_4)$. Since $\Gamma_{2,2}\in\Gamma_{1,2}^2$, there exist $x_1\in P_{(1,2),(1,2)}(x_0,x_2)$, $x_{3}\in P_{(1,2),(1,2)}(x_2,x_4)$ and $y_2\in P_{(1,2),(1,2)}(x_0,x_4)$. Since $x_1,y_2\in N^+(x_0)$, we have $(x_1,y_2)$ or $(y_2,x_1)\in A(\Gamma)$. If $(x_1,y_2)\in A(\Gamma)$, by $x_2,y_2\in N^+(x_1)$, then $(y_2,x_2)$ or $(x_2,y_2)\in A(\Gamma)$, which implies $\{(x_0,x_2),(x_2,x_0),(x_2,x_4),(x_4,x_2)\}\cap A(\Gamma)\neq\emptyset$ since $x_2,x_4\in N^+(y_2)$ or $x_0,x_2\in N^-(y_2)$, a contradiction. If $(y_2,x_1)\in A(\Gamma)$, by $x_1,x_4\in N^+(y_2)$, then $(x_1,x_4)$ or $(x_4,x_1)\in A(\Gamma)$, which implies that $\{(x_0,x_4),(x_4,x_0),(x_2,x_4),(x_4,x_2)\}\cap A(\Gamma)\neq\emptyset$ since $x_2,x_4\in N^+(x_1)$ or $x_0,x_4\in N^-(x_1)$, a contradiction. Thus, our claim is valid.

Since $2\notin T$, from Lemma \ref{delta3}, we have $F_3=\{\Gamma_{0,0},\Gamma_{1,2},\Gamma_{2,1},\Gamma_{2,2}\}$. By Lemma \ref{1,2^2} \ref{1,2 2,1}, one gets $\Gamma_{1,2}\Gamma_{2,1}\subseteq\{\Gamma_{0,0},\Gamma_{1,2},\Gamma_{2,1}\}$. Lemma \ref{jb} \ref{jb-1} implies that $$k_{1,2}^2=k_{1,2}+p_{(1,2),(2,1)}^{(1,2)}k_{1,2}+p_{(1,2),(2,1)}^{(2,1)}k_{1,2}.$$ By Lemma \ref{jb} \ref{jb-2}, we obtain $$p_{(1,2),(1,2)}^{(1,2)}=p_{(1,2),(2,1)}^{(1,2)}=p_{(1,2),(2,1)}^{(2,1)}=(k_{1,2}-1)/2.$$ In view of Lemma \ref{1,2^2} \ref{1,2^2,1}, one has $\Gamma_{1,2}^2\subseteq\{\Gamma_{1,2},\Gamma_{2,1},\Gamma_{2,2}\}$. By Lemma \ref{jb} \ref{jb-1}, we get $$k_{1,2}^2=p_{(1,2),(1,2)}^{(1,2)}k_{1,2}+p_{(1,2),(1,2)}^{(2,1)}k_{1,2}+p_{(1,2),(1,2)}^{(2,2)}k_{2,2}.$$ Then
\begin{align}\label{k22=1-1}
k_{2,2}=k_{1,2}(k_{1,2}+1-2p_{(1,2),(1,2)}^{(2,1)})/(2p_{(1,2),(1,2)}^{(2,2)}).
\end{align}

In view of Lemma \ref{jb} \ref{jb-2}, one has $p_{(2,1),(2,1)}^{(1,2)}=p_{(1,2),(1,2)}^{(2,1)}$ and $p_{(1,2),(1,2)}^{(2,2)}k_{2,2}/k_{1,2}=p_{(2,1),(2,2)}^{(1,2)}$. By setting $\wz{d}=\wz{e}=\wz{f}=\wz{g}^t=(1,2)$ in Lemma \ref{jb} \ref{jb-4}, we get
\begin{align}
(p_{(1,2),(1,2)}^{(2,1)})^2+(p_{(1,2),(1,2)}^{(2,2)})^2k_{2,2}/k_{1,2}=k_{1,2}+(p_{(1,2),(1,2)}^{(1,2)})^2,\nonumber
\end{align}
  which implies $$(k_{1,2}+1-2p_{(1,2),(1,2)}^{(2,1)})(k_{1,2}-2p_{(1,2),(1,2)}^{(2,2)}+1+2p_{(1,2),(1,2)}^{(2,1)})=0$$
  from \eqref{k22=1-1}. Since $k_{2,2}\neq0$, we obtain
 \begin{align}\label{ewai1}
 p_{(1,2),(1,2)}^{(2,2)}=(k_{1,2}+1)/2+p_{(1,2),(1,2)}^{(2,1)}.
 \end{align}
 Since $\Gamma_{2,2}\in F_3$, $\Delta_I(x)$ is not semicomplete for $x\in V(\Gamma)$, where $I=\{r\mid\Gamma_{1,r-1}\in F_3\}$. 
Since $F_3=\{\Gamma_{0,0},\Gamma_{1,2},\Gamma_{2,1},\Gamma_{2,2}\}$, from Lemma \ref{2,2 1} \ref{2,2 1-1}, we have $\Gamma_{2,2}\Gamma_{1,2}\subseteq\{\Gamma_{2,1},\Gamma_{2,2}\}$. By Lemma \ref{jb} \ref{jb-2} and \ref{jb-3}, we have $p_{(1,2),(2,2)}^{(2,2)}=k_{1,2}-p_{(1,2),(1,2)}^{(2,2)}$, and so $p_{(1,2),(2,2)}^{(2,2)}=(k_{1,2}-1)/2-p_{(1,2),(1,2)}^{(2,1)}$ from \eqref{ewai1}. By the claim, one gets $\Gamma_{2,2}^2\subseteq\{\Gamma_{0,0},\Gamma_{1,2},\Gamma_{2,1}\}$.  In view of Lemma \ref{jb} \ref{jb-1} and \ref{jb-2}, one gets $k_{2,2}^2=k_{2,2}+2p_{(1,2),(2,2)}^{(2,2)}k_{2,2}$. \eqref{k22=1-1} implies
  \begin{align}\label{ewai2}
 (k_{1,2}+1-2p_{(1,2),(1,2)}^{(2,1)})(k_{1,2}-1-2p_{(1,2),(1,2)}^{(2,1)})=0.
 \end{align}

(i) Since $k_{2,2}\neq0$, from \eqref{k22=1-1} and \eqref{ewai2}, we get $p_{(1,2),(1,2)}^{(2,1)}=(k_{1,2}-1)/2$. Thus,
\ref{k22-1} holds.

 (ii) Since $p_{(1,2),(1,2)}^{(2,1)}=(k_{1,2}-1)/2$, \ref{k22-2} holds from \eqref{ewai1}.

 (iii) Since  $p_{(1,2),(1,2)}^{(2,1)}=(k_{1,2}-1)/2$, \ref{k22-3} holds from \eqref{k22=1-1} and \ref{k22-2}.
\end{proof}

\begin{prop}\label{delta2,2}
	Let $(1,2)$ be pure and $2\notin T$. Suppose that $4\notin T$ or $(1,3)$ is pure. If $\Gamma_{2,2}\in\Gamma_{1,2}^2$, then $\Delta_3$ is isomorphic to a doubly regular $(k_{1,2}+1,2)$-team tournament of type \Rmnum{2} with $k_{1,2}\equiv 3$ {\rm(mod $4$)}.
\end{prop}
\begin{proof}
	By Lemma \ref{delta3}, one gets $F_3=\{\Gamma_{0,0},\Gamma_{1,2},\Gamma_{2,1},\Gamma_{2,2}\}$. 	Let $x\in V(\Gamma)$. Pick distinct vertices $z,w\in F_3(x)$. Since $(1,2)$ is pure and $\Gamma_{2,2}\in\Gamma_{1,2}^2$, we have $\wz{\partial}_{\Gamma}(z,w)=\wz{\partial}_{\Delta_3(x)}(z,w)$, which implies that	$[\Delta_3(x)]_{\wz{i}}(z)\cap[\Delta_3(x)]_{\wz{j}^t}(w)=P_{\wz{i},\wz{j}}(z,w)$ for $\wz{i},\wz{j}\in\wz{\partial}(\Delta_3(x))$ and $z,w\in F_3(x)$. By the weakly distance-regularity of $\Gamma$, $\Delta_3(x)$ is weakly distance-regular. By Lemma \ref{k22=1} \ref{k22-3}, one obtains $k_{2,2}=1$. It follows from Lemma \ref{k22=1} \ref{k22-1}, \ref{k22-2} and \cite[Proposition 5.3]{JG14} that $\Delta_3(x)$ is isomorphic to a doubly regular $(k_{1,2}+1,2)$-team tournament $\Lambda$ with parameters $((k_{1,2}-1)/2,(k_{1,2}-1)/2,k_{1,2})$. By \cite[Theorem 4.3]{JG14}, $\Lambda$ is of type \Rmnum{2}. \cite[Proof of Theorem 4.6]{JG14} implies that $[(k_{1,2}-1)/2-1]/2$ is an integer. Then $k_{1,2}\equiv 3$ (mod $4$). Thus, the desired result follows.
\end{proof}


\begin{prop}\label{(1,q-1) mixed}
If $(1,q-1)$ is mixed with $q\geq4$, then $\Delta_{\{q-1,q\}}$ is  isomorphic to ${\rm Cay}(\mathbb{Z}_{2q-2},\{1,2\})[\overline{K}_{k_{1,q-1}}]$.
\end{prop}
\begin{proof}
By Corollary \ref{main cor}, $(1,q-2)$ is pure. Pick a circuit $(x_0,x_2,\ldots,x_{2q-4})$ consisting of arcs of type $(1,q-2)$. Since $p_{(1,q-1),(1,q-1)}^{(1,q-2)}\neq0$, there exist vertices $x_1,x_3,\ldots,x_{2q-3}$ such that $(x_0,x_1,\ldots,x_{2q-3})$ is a circuit consisting of arcs of type $(1,q-1)$, where the indices are read modulo $2q-2$. Since $\Gamma_{1,q-1}^2=\{\Gamma_{1,q-2}\}$ from Corollary \ref{main cor}  \ref{main cor-1}, one has $(x_i,x_{i+2})\in\Gamma_{1,q-2}$ for all integer $i$. Since $p_{(1,q-1),(1,q-1)}^{(1,q-2)}=k_{1,q-1}$ from Lemma \ref{comm} and Corollary \ref{main cor} \ref{main cor-2}, we get $P_{(1,q-1),(1,q-1)}(x_{i-1},x_{i+1})=\Gamma_{1,q-1}(x_{i-1})$. By setting $\wz{d}=\wz{e}=(1,q-1)$ in Lemma \ref{jb} \ref{jb-1}, we have $k_{1,q-2}=k_{1,q-1}$. Since $\Gamma_{1,q-1}^2=\{\Gamma_{1,q-2}\}$ again, one obtains $(y_i,y_{i+2})\in\Gamma_{1,q-2}$  for all $y_i\in P_{(1,q-1),(1,q-1)}(x_{i-1},x_{i+1})$ and $y_{i+2}\in P_{(1,q-1),(1,q-1)}(x_{i+1},x_{i+3})$. Since  $p_{(1,q-1),(1,q-1)}^{(1,q-2)}=k_{1,q-1}$ and $(y_i,x_{i+2})\in\Gamma_{1,q-2}$, one gets $(y_i,y_{i+1})\in\Gamma_{1,q-1} $ for all $y_i\in P_{(1,q-1),(1,q-1)}(x_{i-1},x_{i+1})$ and $y_{i+1}\in P_{(1,q-1),(1,q-1)}(x_{i},x_{i+2})$. Thus, the desired result follows.
\end{proof}

\section{Proof of Theorem \ref{main1}}

In this section, we always assume that $\Gamma$ is a locally semicomplete commutative weakly distance-regular digraph. To give a proof of Theorem \ref{main1}, we need three auxiliary lemmas.

\begin{lemma}\label{quotient}
	Let $T=\{2,3\}$. Suppose that $(1,2)$ is pure. If $\Gamma_{2,2}\in\Gamma_{1,2}^2$, then $\Gamma$ is isomorphic to $\Sigma\circ K_{k_{1,1}+1}$, where $\Sigma$ is a locally semicomplete weakly distance-regular digraph with $\wz{\partial}(\Sigma)=\{(0,0),(1,2),(2,1),(2,2)\}$.
\end{lemma}
\begin{proof}
	By Lemma \ref{delta3}, one gets $\wz{\partial}(\Gamma)=\{(0,0),(1,1),(1,2),(2,1),(2,2)\}$. Since $(1,2)$ is pure, we have $\Gamma_{1,2}\notin\Gamma_{1,1}^2$. It follows from Lemma \ref{1,1} that $F_2=\{\Gamma_{0,0},\Gamma_{1,1}\}$.
	
	We define a digraph $\Sigma$ with vertex set $\{F_2(x)\mid x\in V(\Gamma)\}$ in which $(F_2(x),F_2(y))$ is an arc whenever there is an arc in $\Gamma$ from $F_2(x)$ to $F_2(y)$.
	
	We claim that $(x',y')\in\Gamma_{1,2}$ for $x'\in F_2(x)$ and $y'\in F_2(y)$ when $(x,y)\in\Gamma_{1,2}$. Without loss of generality, we may assume $x'\neq x$. Since $F_2=\{\Gamma_{0,0},\Gamma_{1,1}\}$, we have $(x,x')\in\Gamma_{1,1}$. In view of Lemma \ref{p(1,2),(2,1)^(1,1)=k_1,2}, one gets $y\in P_{(1,2),(2,1)}(x,x')$. If $y=y'$, then our claim is valid. If $y\neq y'$, then $(y,y')\in\Gamma_{1,1}$, and so $x'\in P_{(2,1),(1,2)}(y,y')$ from Lemma \ref{p(1,2),(2,1)^(1,1)=k_1,2}. Thus, our claim is valid. By Lemma \ref{1,1},  $\Gamma$ is isomorphic to $\Sigma\circ K_{k_{1,1}+1}$.
	
	Let $x$ and $y$ be vertices with $y\notin F_2(x)$. Pick a shortest path $P:=(x=x_0,x_1,\ldots,x_l=y)$ from $x$ to $y$ in $\Gamma$. Since $y\notin F_2(x)$, $P$ contains an arc of type $(1,2)$. Suppose that $P$ contains an edge. By the commutativity of $\Gamma$, we may assume that $(x_0,x_1)\in\Gamma_{1,1}$ and $(x_{1},x_2)\in\Gamma_{1,2}$. By Lemma \ref{p(1,2),(2,1)^(1,1)=k_1,2}, one has $(x_0,x_2)\in\Gamma_{1,2}$, a contradiction. Then $P$ consists of arcs of type $(1,2)$. It follows that $(F_2(x_0),F_2(x_1),\ldots,F_2(x_{l}))$ is a path in $\Sigma$, and so $\partial_{\Sigma}(F_2(x),F_2(y))\leq\partial_{\Gamma}(x,y)$.
	
	Pick a shortest path $(F_2(x)=F_2(y_0),F_2(y_1),\ldots,F_2(y_{h})=F_2(y))$ from $F(x)$ to $F(y)$ in $\Sigma$. It follows that there exists $y_i'\in F_2(y_i)$ for each $i\in\{0,1\ldots,h\}$ such that $(y_0',y_1',\ldots,y_h')$ is a path consisting of arcs of type $(1,2)$ in $\Gamma$. By the claim,  $(x,y_1',\ldots,y_{h-1}',y)$ is a path in $\Gamma$. It follows that $\partial_{\Gamma}(x,y)\leq\partial_{\Sigma}(F_2(x),F_2(y))$. Thus, $\partial_{\Gamma}(x,y)=\partial_{\Sigma}(F_2(x),F_2(y))$ and $\wz{\partial}(\Sigma)=\{(0,0),(1,2),(2,1),(2,2)\}$.
	
	Let $\wz{h}\in\wz{\partial}(\Sigma)$ and $(F_2(u),F_2(v))\in\Sigma_{\wz{h}}$. Then $(u,v)\in \Gamma_{\wz{h}}$. For $\wz{i},\wz{j}\in\wz{\partial}(\Sigma)$, we have
	$$P_{\wz{i},\wz{j}}(u,v)=\cup_{w\in P_{\wz{i},\wz{j}}(u,v)}F_2(w)
	=\cup_{F_2(w)\in \Sigma_{\wz{i}}(F_2(u))\cap\Sigma_{\wz{j}^t}(F_2(v))}F_2(w).$$
	Since $F_2=\{\Gamma_{0,0},\Gamma_{1,1}\}$, one gets $|F_2(w)|=k_{1,1}+1$, which implies $$|\Sigma_{\wz{i}}(F_2(u))\cap\Sigma_{\wz{j}^t}(F_2(v))|=p_{\wz{i},\wz{j}}^{\wz{h}}/(k_{1,1}+1).$$ By $\Gamma$ is a commutative weakly distance-regular digraph, $\Sigma$ is a  commutative weakly distance-regular digraph. Since $\Gamma$ is locally semicomplete, $\Sigma$ is also locally semicomplete.
	
	This completes the proof of this lemma.
\end{proof}

\begin{lemma}\label{type}
Let $(1,q-1)$ be pure with $q\geq4$. Suppose that $r\in T$ with $r\neq q$. The one of the following hold:
\begin{enumerate}
	\item\label{type-2} $r\in\{2,3\}$ and $\Gamma_{1,r-1}\Gamma_{1,q-1}=\{\Gamma_{1,q-1}\}$;
	
\item\label{type-1} $r=q+1$ and {\rm C}$(q+1)$ exists.
\end{enumerate}
\end{lemma}
\begin{proof}
By Lemma \ref{pure}, there exists a path $(x,y,z)$ consisting of arcs of type $(1,q-1)$ with $(x,z)\in\Gamma_{2,q-2}$. Pick a vertex $w\in\Gamma_{1,r-1}(x)$. Since $y,w\in N^+(x)$, we have $(w,y)$ or $(y,w)\in A(\Gamma)$. If $(y,w)\in A(\Gamma)$, by $w,z\in N^+(y)$, then $(w,z)\in\Gamma_{1,s}$ for some $s>1$ since $(x,z)\in\Gamma_{2,q-2}$ with $q\geq4$, which implies $w\in P_{(1,r-1),(1,s)}(x,z)$, contrary to Lemma \ref{not pure}. Then $(w,y)\in\Gamma_{1,s-1}$ with $s>1$.

If $s=q$, then $y\in P_{(1,q-1),(q-1,1)}(x,w)$, which implies that \ref{type-2} holds from Lemmas \ref{pure} and \ref{xingjia}. We only need to consider the case $s\neq q$.

If $s=r$, then $\Gamma_{1,q-1}\in\Gamma_{1,r-1}^2$. Suppose $s\neq r$. Since $w\in P_{(1,r-1),(1,s-1)}(x,y)$, there exists $w'\in P_{(1,s-1),(1,r-1)}(x,y)$. Since $(1,q-1)$ is pure, from Lemma \ref{not pure}, we have $\partial(w,z)=\partial(w',z)=2$. Since $q-1\leq\partial(z,w)\leq1+\partial(z,x)$ and $q-1\leq\partial(z,w')\leq1+\partial(z,x)$, one gets $(w,z),(w',z)\in\Gamma_{2,q-1}$. The fact $y\in P_{(1,r-1),(1,q-1)}(w',z)$ implies that there exists $y'\in P_{(1,r-1),(1,q-1)}(w,z)$. By Lemmas \ref{comm} and \ref{pure}, we have $y'\in P_{(1,q-1),(1,q-1)}(x,z)$. Since $w\in P_{(1,r-1),(1,r-1)}(x,y')$, we obtain $\Gamma_{1,q-1}\in\Gamma_{1,r-1}^2$. Thus, we conclude that $\Gamma_{1,q-1}\in\Gamma_{1,r-1}^2$.

Suppose $r\leq 3$. Since $\Gamma_{1,q-1}\in\Gamma_{1,r-1}^2$, we have $3\leq q-1\leq 2r-2$, which implies $r=3$. Since $\Gamma_{1,q-1}\in\Gamma_{1,2}^2$, one gets $\Gamma_{1,q-1}\in F_{\{2,3\}}$. By Lemma \ref{(1,2) mixed}, $(1,2)$ is pure. It follows from Lemma \ref{1,2^2} \ref{1,2^2,1} that $q=4$, contrary to Lemma \ref{delta3}. 
Thus, $r\geq 4$. Since $\Gamma_{1,q-1}\in\Gamma_{1,r-1}^2$ again, from Lemma \ref{pure}, $(1,r-1)$ is mixed. By Corollary \ref{main cor}, $r=q+1$ and C$(q+1)$ exists. Thus, \ref{type-1} holds.
\end{proof}

\begin{lemma}\label{necessity}
All digraphs in Theorem \ref{main1} are locally semicomplete commutative weakly distance-regular digraphs.
\end{lemma}
\begin{proof}
Routinely, all digraphs in Theorem \ref{main1} are locally semicomplete. Let $\Lambda$ be a doubly regular $(r,2)$-team tournament of type \Rmnum{2} for a positive integer $r$ with $4\mid r$, where $V(\Lambda)=V_1\cup \cdots \cup V_r$ is the partition of the vertex set into $r$ independent sets of size $2$. Since $|N^+(x)\cap V_i|=1$ for all $x\notin V_i$ and $i\in \{1,\cdots,r\}$, from \cite[Proof of Theorem 4.6]{JG14}, we have $\tilde{\partial}(\Lambda)=\{(0,0),(1,2),(2,1),(2,2)\}$. By \cite[Theorem 5.5]{JG14}, $(V(\Lambda),\{\Lambda_{0,0},\Lambda_{1,2},\Lambda_{2,1},\Lambda_{2,2}\})$ is an association scheme. \cite[Theorem 3.1]{KSW03} implies that $\Lambda$ is weakly distance-regular. By \cite[Proposition 2.6 and Theorem 1.1]{KSW03}, all digraphs in Theorem \ref{main1} \ref{main1-4} and \ref{main1-1} are commutative weakly distance-regular digraphs. It suffices to show that all digraphs in Theorem \ref{main1} \ref{main1-2} are weakly distance-regular.

Let $\Gamma''=\Gamma'\circ(\Sigma_x)_{x\in\mathbb{Z}_{iq}}$, where $\Gamma'={\rm Cay}(\mathbb{Z}_{iq},\{1,i\})$ for some $i\in\{1,2\}$, $q\geq4$, and $(\Sigma_x)_{x\in \mathbb{Z}_{iq}}$ are all semicomplete weakly distance-regular digraphs with $p_{\wz{i},\wz{j}}^{\wz{h}}(\Sigma_0)=p_{\wz{i},\wz{j}}^{\wz{h}}(\Sigma_x)$ for each $x$ and $\wz{i},\wz{j},\wz{h}$. It follows that $|V(\Sigma_x)|=t$ and $\tilde{\partial}(\Sigma_x)=J$ with $J\subseteq\{(0,0),(1,1),(1,2),(2,1)\}$ for all $x\in\mathbb{Z}_{iq}$. Since $q\ge 4$, one has $\tilde{\partial}(\Gamma')\cap J=\{(0,0)\}$. Since $\Gamma''=\Gamma'\circ(\Sigma_x)_{x\in\mathbb{Z}_{iq}}$, we get
\begin{align}
\tilde{\partial}_{\Gamma''}((x,u_x),(y,v_y))&=\tilde{\partial}_{\Gamma'}(x,y)~{\rm for~all}~(x,u_x),(y,v_y)\in V(\Gamma'')~{\rm with}~x\neq y\label{nec-1},\\
\tilde{\partial}_{\Gamma''}((x,u_x),(x,v_x))&=\tilde{\partial}_{\Sigma_x}(u_x,v_x)~{\rm for~all}~(x,u_x),(x,v_x)\in V(\Gamma'').\label{nec-2}
\end{align}
Then $\tilde{\partial}(\Gamma'')=\tilde{\partial}(\Gamma')\cup J$.  It suffices to show that $|P_{\wz{i},\wz{j}}((x,u_x),(y,v_y))|$ only depends on $\tilde{i},\tilde{j},\tilde{\partial}((x,u_x),(y,v_y))\in\tilde{\partial}(\Gamma')\cup J$.

Suppose  $\tilde{\partial}((x,u_x),(y,v_y))\in J$. Since  $\tilde{\partial}(\Gamma')\cap J=\{(0,0)\}$, from \eqref{nec-1}, we get $x=y$.   If $\{\tilde{i},\tilde{j}\}\cap J\ne\emptyset$, by \eqref{nec-2}, then $\{\tilde{i},\tilde{j}\}\nsubseteq J$ and $P_{\wz{i},\wz{j}}((x,u_x),(x,v_x))=\emptyset$, or $\{\tilde{i},\tilde{j}\}\subseteq J$ and $$|P_{\wz{i},\wz{j}}((x,u_x),(x,v_x))|=|\{z_x\mid \tilde{\partial}_{\Sigma_x}(u_x,z_x)=\wz{i},\tilde{\partial}_{\Sigma_x}(z_x,v_x)=\wz{j}\}|.$$  If $\{\tilde{i},\tilde{j}\}\cap J=\emptyset$,    then $\tilde{i},\tilde{j}\in \tilde{\partial}(\Gamma')\backslash\{(0,0)\}$ since $\tilde{\partial}(\Gamma')\cap J=\{(0,0)\}$, which implies  $\tilde{j}\neq\tilde{i}^t$ and $P_{\wz{i},\wz{j}}((x,u_x),(x,v_x))=\emptyset$, or $\tilde{j}=\tilde{i}^t$ and $$|P_{\wz{i},\wz{i}^t}((x,u_x),(x,v_x))|=t\cdot|\{y\in V(\Gamma')\mid(x,y)\in\Gamma_{\tilde{i}}'\}|$$  by \eqref{nec-1}, where $t=|V(\Sigma_z)|$ for all $z\in\mathbb{Z}_{iq}$. Since $\Gamma'$ is weakly distance-regular from \cite[Theorem 1.1]{KSW03} and $\Sigma_x$ is weakly distance-regular, $|P_{\wz{i},\wz{j}}((x,u_x),(y,v_y))|$ only depends on $\tilde{i},\tilde{j},\tilde{\partial}((x,u_x),(y,v_y))$.

Suppose $\tilde{\partial}((x,u_x),(y,v_y))\notin J$.  Since $\tilde{\partial}(\Gamma')\cap J=\{(0,0)\}$, from \eqref{nec-2}, we have $x\neq y$ and $\tilde{\partial}((x,u_x),(y,v_y))\in \tilde{\partial}(\Gamma')\backslash\{(0,0)\}$. If $\{\tilde{i},\tilde{j}\}\cap \tilde{\partial}(\Gamma')= \emptyset$, then $\{\tilde{i},\tilde{j}\}\subseteq J\backslash\{(0,0)\}$, which implies $P_{\wz{i},\wz{j}}((x,u_x),(y,v_y))=\emptyset$ by \eqref{nec-1}. Assume that $\{\tilde{i},\tilde{j}\}\cap \tilde{\partial}(\Gamma')\ne\emptyset$. If $\tilde{i},\tilde{j}\in \tilde{\partial}(\Gamma')$, by \eqref{nec-1} and \eqref{nec-2}, then $$|P_{\wz{i},\wz{j}}((x,u_x),(y,v_y))|=t\cdot|\{z\in V(\Gamma')\mid (x,z)\in\Gamma_{\wz{i}}',(z,y)\in\Gamma_{\wz{j}}'\}|,$$ where $t=|V(\Sigma_z)|$ for all $z\in\mathbb{Z}_{iq}$. Now we consider the case $|\{\tilde{i},\tilde{j}\}\cap \tilde{\partial}(\Gamma')|=1$. Since the proofs are similar, we may assume $\tilde{i}\in J$ and $\tilde{j}\in\wz{\partial}(\Gamma')$. If $\tilde{\partial}((x,u_x),(y,v_y))\neq\tilde{j}$, from \eqref{nec-1} and \eqref{nec-2}, then $P_{\wz{i},\wz{j}}((x,u_x),(y,v_y))=\emptyset$; if $\tilde{\partial}((x,u_x),(y,v_y))=\tilde{j}$, then $$|P_{\wz{i},\wz{j}}((x,u_x),(y,v_y))|=|\{(x,z_x)\mid\tilde{\partial}_{\Sigma_x}(u_x,z_x)=\tilde{i}\}|.$$  Since $\Gamma'$ is weakly distance-regular from \cite[Theorem 1.1]{KSW03} and $\Sigma_x$ is weakly distance-regular, $|P_{\wz{i},\wz{j}}((x,u_x),(y,v_y))|$ only depends on $\tilde{i},\tilde{j},\tilde{\partial}((x,u_x),(y,v_y))$.

Thus, the desired result follows.
\end{proof}
\begin{proof}[Proof of Theorem~\ref{main1}]
The sufficiency is immediate from Lemma \ref{necessity}. We now prove the necessity. Let $q=\max T$. We divide the proof into two cases according to whether $(1,q-1)$ is pure.

\textbf{Case 1.} $(1,q-1)$ is pure.

Since $\Gamma$ is not semicomplete, we have $q\geq3$.

\textbf{Case 1.1.} $q=3$.

By Lemma \ref{p(1,2),(2,1)^(1,1)=k_1,2}, we obtain $2\notin T$ or $\Gamma_{1,1}\in F_3$. Since $\Gamma$ is not semicomplete, we have $F_3\nsubseteq\{\Gamma_{0,0},\Gamma_{1,1},\Gamma_{1,2},\Gamma_{2,1}\}$ from Lemma \ref{diameter 2}. By Lemma \ref{delta3}, we have $\Gamma_{2,2}\in\Gamma_{1,2}^2$, which implies $F_3=\{\Gamma_{0,0},\Gamma_{1,2},\Gamma_{2,1},\Gamma_{2,2}\}$ or $\{\Gamma_{0,0},\Gamma_{1,1},\Gamma_{1,2},\Gamma_{2,1},\Gamma_{2,2}\}$.

If $2\notin T$, from Proposition \ref{delta2,2}, then $\Gamma$ is isomorphic to the digraph in \ref{main1-4} with $m=1$; if $2\in T$, from Proposition \ref{delta2,2} and Lemma \ref{quotient} , then $\Gamma$ is isomorphic to the digraph in \ref{main1-4} with $m=k_{1,1}+1$ and $r=k_{1,2}+1$.

\textbf{Case 1.2.} $q\geq4$.

By Lemma \ref{pure}, there exists an isomorphism $\sigma$ from ${\rm Cay}(\mathbb{Z}_q,\{1\})[\overline{K}_m]$ to $\Delta_q(x)$ with $m=k_{1,q-1}$. In view of Lemma \ref{type}, we have $T\subseteq\{2,3,q\}$. Since $\Gamma$ is locally semicomplete, we have
\begin{align}
\Gamma_{1,q-1}\Gamma_{q-1,1}\setminus\{\Gamma_{0,0}\}=\cup_{a\in T\setminus\{q\}}\{\Gamma_{1,a-1},\Gamma_{a-1,1}\},\label{zhudingli1}
\end{align}
which implies $F_q(x)=V(\Gamma)$. If $m=1$, then $\Gamma$ is isomorphic to the digraph in \ref{main1-1} with $i=1$. Now suppose $m>1$. By \eqref{zhudingli1}, one gets $\{q\}\subsetneq T$. Since $(\sigma(i,j)),(\sigma(i+1,j'))\in\Gamma_{1,q-1}$ for $j,j'\in\mathbb{Z}_m$, from \eqref{zhudingli1}, we obtain $F_{T\setminus\{q\}}(\sigma(i,0))=\{\sigma(i,j)\mid j\in\mathbb{Z}_m\}$. By Lemmas \ref{bkb} and \ref{type}, $\Delta_{T\setminus\{q\}}(\sigma(i,0))$ for all $i\in\mathbb{Z}_q$ are  semicomplete weakly distance-regular digraphs with the same intersection numbers. Thus, $\Gamma$ is isomorphic to the digraph in \ref{main1-1} with $i=1$ or \ref{main1-2} with $i=1$.

\textbf{Case 2.} $(1,q-1)$ is mixed.

Suppose $q=3$. Then $T=\{2,3\}$, and $\Gamma_{1,2}\in\Gamma_{1,1}^2$ or $\Gamma_{1,1}\in\Gamma_{1,2}^2$. By Lemma \ref{(1,2) mixed}, $\Gamma$ is semicomplete, a contradiction. Then $q\geq4$.

By the first statement of Corollary \ref{main cor}, C$(q)$ exists. Then $\{q,q-1\}\subseteq T$. By Proposition \ref{(1,q-1) mixed}, there exists an isomorphism $\sigma$ from ${\rm Cay}(\mathbb{Z}_{2q-2},\{1,2\})[\overline{K}_m]$ to $\Delta_{\{q-1,q\}}(x)$ with $m=k_{1,q-1}$.


In view of Lemma \ref{type}, we have $T\subseteq\{2,3,q-1,q\}$. Since $\Gamma$ is locally semicomplete, from Lemma \ref{mixed-3} \ref{mixed-3-1} we have
\begin{align}
\Gamma_{1,q-1}\Gamma_{q-1,1}\setminus\{\Gamma_{0,0}\}=\Gamma_{1,q-2}\Gamma_{q-2,1}\setminus\{\Gamma_{0,0}\}=\cup_{a\in T\setminus\{q\}}\{\Gamma_{1,a-1},\Gamma_{a-1,1}\},\label{zhudingli2}
\end{align}
which implies $F_q(x)=V(\Gamma)$. If $m=1$, then $\Gamma$ is isomorphic to the digraph in \ref{main1-1} with $l=q-1$ and $i=2$. Now suppose $m>1$. By \eqref{zhudingli2}, one gets $\{q-1,q\}\subsetneq T$. Since $((\sigma(i,j)),(\sigma(i+1,j')))\in\Gamma_{1,q}$ for $j,j'\in\mathbb{Z}_m$, from \eqref{zhudingli2}, we obtain $F_{T\setminus\{q-1,q\}}(\sigma(i,0))=\{\sigma(i,j)\mid j\in\mathbb{Z}_m\}$. By Lemmas \ref{bkb} and \ref{type}, $\Delta_{T\setminus\{q-1,q\}}(\sigma(i,0))$ for all $i\in\mathbb{Z}_{2q-2}$ are semicomplete weakly distance-regular digraphs with the same intersection numbers. If $q=4$, then $T=\{2,3,4\}$, which implies that $\Gamma$ is isomorphic to the digraph in \ref{main1-1} with $l=3$ and $i=2$; if $q>4$, then $\Gamma$ is isomorphic to one of the digraphs in \ref{main1-1} with $i=2$ and \ref{main1-2} with $i=2$.

This completes the proof of Theorem \ref{main1}.
\end{proof}

\section*{Acknowledgements}

 We  are grateful to Professor Hiroshi Suzuki for drawing our attention to \cite{JG14} and for valuable suggestions. Y.~Yang is supported by NSFC (12101575, 52377162) and the Fundamental Research Funds for the Central Universities (2652019319), K. Wang is supported by the National Key R$\&$D Program of China (No.~2020YFA0712900) and NSFC (12071039, 12131011).

\section*{Data Availability Statement}

No data was used for the research described in the article.

\end{document}